\documentclass[10pt]{article}
\usepackage{layout,yfonts, graphicx, amsmath, amsthm, amssymb, euscript, enumerate, bbold, bbm, mathrsfs}
\usepackage[numbers]{natbib} %cambia el estilo citar la bibliografia
\numberwithin{equation}{section}
\usepackage[hang,small,sf]{caption2} 

\usepackage[margin=1in]{geometry}
\theoremstyle{plain}
  \newtheorem{teor}{Theorem}[section]
  
  \newtheorem{prop}[teor]{Proposition}
  \newtheorem{lema}[teor]{Lemma}
  \newtheorem{coro}[teor]{Corollary}
\theoremstyle{remark}  
  \newtheorem{rem}[teor]{Remark}
\theoremstyle{plain}

\newtheoremstyle{hyp}{}{}{\itshape}{}{}{}{0pt}{}
\theoremstyle{hyp}

\DeclareMathAlphabet{\mathpzc}{OT1}{pzc}{m}{it}
\DeclareMathOperator*{\limess}{lim\, inf\, ess}
\DeclareMathOperator{\deri}{D}

\DeclareMathOperator{\Lp}{L}

\DeclareMathOperator{\comp}{c} 
\DeclareMathOperator{\loc}{loc}
\DeclareMathOperator{\hol}{C}
\DeclareMathOperator{\tr}{tr}
\DeclareMathOperator{\sob}{W}
\DeclareMathOperator{\I}{i}
\DeclareMathOperator{\expo}{e}
\DeclareMathOperator{\sop}{supp}
\DeclareMathOperator{\inted}{\mathcal{I}}

\newcommand{\E}{\mathbbm{E}}
\newcommand{\set}{\mathcal{O}}
\newcommand{\dif}{\mathcal{L}}
\newcommand{\inte}{\mathcal{I}}

\newcommand{\ind}{\mathcal{D}}

\newcommand{\R}{\mathbbm{R}}

\newcommand{\uno}{\mathbbm{1}}
\newcommand{\der}{\mathrm{d}}
\newcommand{\F}{\mathbbm{F}}
\newcommand{\Pro}{\mathbbm{P}}
\newcommand{\BE}{\begin{equation}}

\newcommand{\EE}{\end{equation}}
\newcommand {\BA}{\begin{align}}
\newcommand{\EA}{\end{align}}
\newcommand{\eqdef}{\raisebox{0.4pt}{\ensuremath{:}}\hspace*{-1mm}=}
\newcommand{\defeq}{=\hspace*{-1mm}\raisebox{0.4pt}{\ensuremath{:}}}

\title{Solution to HJB equations with an elliptic integro-differential operator and gradient constraint}
\author{Harold A. Moreno-Franco\footnote{National Research University Higher School of Economics, Moscow. Russia.  \textsf{hmoreno@hse.ru}}}
\date{}
\begin{document}	
%\listoftodos %contenido de las notas	
%\layout
%\fontfamily{ptm}\selectfont
\maketitle
\begin{abstract}
\noindent The main goal of this paper is to establish existence, regularity and uniqueness results for the solution of a Hamilton-Jacobi-Bellman (HJB) equation, whose  operator is an elliptic integro-differential operator. The HJB equation studied in this work arises in singular stochastic control problems where the state process is a controlled $d$-dimensional L\'evy process.
\end{abstract}

\section{Introduction and main results}
Our main goal  is to establish the existence,  uniqueness  and regularity of the solution $u$  to the  HJB equation
\begin{equation}\label{p1}	
\begin{cases}
  \max\{qu- \Gamma u-h,|\deri^{1} u|^{2}-1\}=0, & \text{a.e. in}\ B_{R}(0),\\
  u=0,&\text{on}\ \partial B_{R}(0),
\end{cases}  
\end{equation}
where $B_{R}(0)\subset \R^{d}$, with $R>0$ and $d\geq2$ fixed. The components of this equation are: 
\begin{enumerate}[(i)]
\item A constant $q>0$ and a positive function $h:\overline{B_{R}(0)}\longrightarrow\R$.

\item An integro-differential operator $\Gamma$, which has two parts,  an elliptic partial differential operator and an integral operator, i.e.
\begin{align}\label{p6}
  \Gamma u(x)&\eqdef\frac{1}{2}\tr(\sigma\deri^{2} u(x))+\langle\deri^{1} u(x),\gamma\rangle+\int_{\R^{*}}(E(u)(x+z)-u(x)-\langle\deri^{1}u(x),z\rangle)\nu(\der z),
\end{align}
with $x\in B_{R}(0)$.  Here $|\cdot|$ is the Euclidean norm, $\langle\cdot,\cdot\rangle$ is the inner product, $\tr(\cdot)$ is the trace of matrix, $\deri^{1}u=(\partial_{1}u,\dots,\partial_{d}u)$, $\deri^{2}u=(\partial^{2}_{ij}u)_{d\times d}$, $\gamma=(\gamma_{1} ,\dots,\gamma_{d})\in\R^{d}$,  $\sigma=(\sigma_{ij})_{d\times d}\in\R^{d\times d}$ is a positive definite matrix, and $\nu$ is a measure in $\R^{*}\eqdef\R^{d}\setminus\{0\}$ such that $\int_{\R^{*}}\nu(\der z)<\infty$ and $\int_{\R^{*}}|z|\nu(\der z)<\infty$. The  operator $E:\hol^{k,\alpha}(\overline{B_{R}(0)})\longrightarrow\hol^{k,\alpha}_{\comp}(\R^{d})$,  with $k\geq0$ and $\alpha\in[0,1]$, is a continuous linear operator that has the following properties: there exist constants $C=C(k,R)>0$ and $b>0$ such that for every $w\in\hol^{k,\alpha}(\overline{B_{R}(0)})$, 
\begin{equation}\label{p13.1.0}
\begin{cases}
E(w)\,\bigr |_{\overline{B_{R}(0)}}=w,\\
\sop[E(w)]\  \text{is compact,}\\  
\sop[E(w)]\subset B_{R+\frac{b}{2}}(0),\\
||E(w)||_{\hol^{k,\alpha}(\R^{d})}\leq C||w||_{\hol^{k,\alpha}(\overline{B_{R}(0)})},
\end{cases}
\end{equation}
where $\sop[E(w)]\eqdef\{x\in\R^{d}:E(w)(x)\neq0 \}$. The norm $||\cdot||_{\hol^{k,\alpha}(\cdot)}$ is as in (\ref{sp7}) below. This operator is needed to give a sense to the integral term of (\ref{p6}) because $u$ is only defined in $\overline{B_{R}(0)}$. Further details about $E$ will be given in Section \ref{Ch2}.
 \end{enumerate}

Before describing the motivation for studying this equation, let us make it clear that notations of functions spaces in the paper are standard and can be consulted  at  the end of this section.

The stochastic control problem related  to the HJB equation (\ref{p1}) on the whole space, is given when the controlled process $Y=\{Y_{t}:t\geq0\}$ is a $d$-dimensional L\'evy process whose components are a Brownian motion with drift and a Poisson compound process; see Subsection \ref{prob1} below.  Recall that  a L\'evy process is a c\`adl\`ag process  with independent and stationary increments. For background on L\'evy processes we refer to  \cite{sato}, which will be our main reference. 

Since  $\int_{\R^{*}}|z|\nu(\der z)<\infty$ and the continuous linear operator $E$ satisfies (\ref{p13.1.0}), we see that $\Gamma$, given in (\ref{p6}), can be written as
\begin{align}\label{p1.0.5}
  \Gamma u(x)&=\frac{1}{2}\tr(\sigma\deri^{2} u(x))+\langle\deri^{1} u(x),\widetilde{\gamma}\rangle+\int_{\R^{*}}(E(u)(x+z)-u(x))\nu(\der z),\ \text{for all}\ x\in B_{R}(0),
\end{align}
where 
\begin{equation}\label{p.1.0.1}
\widetilde{\gamma}\eqdef\gamma-\int_{\R^{*}}z\nu(\der z).
\end{equation}
The  operator $\Gamma$ corresponds  to the infinitesimal generator of a $d$-dimensional L\'evy process  $Y=\{Y_{t}:t\geq0\}$ given as in (\ref{p8.1}). A simple example of continuous L\'evy processes is a $d$-dimensional standard Brownian motion. 

When the controlled process  $Y$ is a two-dimensional standard Brownian motion, Soner and Shreve \cite{soner} showed that the value function related to this singular stochastic control problem, satisfies the following HJB equation
\begin{equation}\label{HJB1.0}
\max\{u-\Delta u-h,|\deri^{1}u|^{2}-1\}=0,
\end{equation}
where $\Delta u\eqdef\partial^{2}_{11}u+\partial^{2}_{22}u$, and $h\in \hol^{2,1}_{\loc}(\R^{2})$ is a strictly convex function, which  has some properties of polynomial growth.  Soner and Shreve \cite{soner} proved that there exists a unique solution $u\in\hol^{2,\alpha}_{\loc}(\R^{2})$ to  the  equation (\ref{HJB1.0}), which is a non-negative convex function. Also, they showed that the value function given in (\ref{p1.0.3}) satisfies the HJB equation (\ref{HJB1.0}), when the controlled process is a two-dimensional standard Brownian motion. In case that the controlled process is a $d$-dimensional standard Brownian motion, with $d>2$, Kruk \cite{kruk} showed that the value function of this stochastic control problem is related to the solution of the HJB equation (\ref{HJB1.0}), with $h:\R^{d}\longrightarrow\R$ and $\Delta u\eqdef\sum_{i=1}^{d}\partial^{2}_{ii}u$. In this case, the solution to the equation (\ref{HJB1.0}) is in $\sob^{2,\infty}_{\loc}(\R^{d})$; see \cite{menaldi}.  In our setting the controlled process is allowed to be a more general $d$-dimensional L\'evy process, it has a continuous component given by a Brownian motion with drift and a component with jumps given by  a compound Poisson process, whose jumps occur at  exponential times  with parameters  $\nu(\R^{*})$ and jump sizes distributed as  $\nu(\R^{*})^{-1}\nu(\der z)$. This makes that the HJB equation, given in (\ref{p1}) on the whole space, differs from (\ref{HJB1.0}) by an integral term coming from the compound Poisson process in the controlled process, whose infinitesimal generator is closely related to the integral term of (\ref{p1.0.5}).

A closely related problem has been recently considered by Menaldi and Robin \cite{menal2}. There they  studied a singular control problem for a multidimensional Gaussian-Poisson process, and  announced a relationship between the value function to this problem and the solution  to the corresponding HJB equation, on the whole space, which is similar to \eqref{HJB1.0}. The multidimensional Gaussian-Poisson process is a L\'evy process where it only has  a $d$-dimensional standard Brownian motion and a jump process whose L\'evy measure $\nu$ satisfies $\int_{\R^{*}}|z|^{p}\nu(\der z)<\infty$, for all $p\geq2$. There, the main arguments to justify that the value function is the solution to the HJB equation a.e.,  and is a twice  weakly differentiable function on $\R^{d}$, have been highlighted. 

The main difference between Menaldi and Robin work \cite{menal2} and ours, is that we consider a HJB equation on a ball; see \eqref{p1},whilst in the former the equation is defined in whole space. In this setting  we establish the existence, regularity and uniqueness of the solution $u$ of \eqref{p1}. Among other things, this requires the inclusion of the operator $E$, given in \eqref{p13.1.0}, to give a sense to the integral term in \eqref{p1.0.5}, because $u$ is only defined in $\overline{B_{R}(0)}$.      

Further than the mathematical motivation for studying the above described problem, this is intimately related to a singular control problem, described in Subsection \ref{prob1}, which arises in risk theory. Indeed, an important problem in risk theory   is to determine an optimal dividend payment strategy for an insurance company to pay to its shareholders. This problem has been intensively studied when the insurance company's surplus is modelled by a unidimensional process. For example, when this is described by a Brownian motion with drift \cite{ger, jean, tak}; a Cramer-Lundberg process  \cite{az, bu, ger, sc, zho}; or a diffusion process \cite{asm, asta}. Recently, it has been analyzed in the more general case when the risk process is a spectrally negative L\'evy process \cite{app1, kipal, loe, rena, zho2}, i.e. a L\'evy process which L\'evy measure is supported in $(-\infty,0)$ \cite{kyp}.  The results obtained in this paper and its applications to risk theory are the topic of a work in progress by the author.

A classical approach to guarantee the existence and regularity of the HJB equation (\ref{p1}), when the operator $\Gamma$ has only the differential term,  consists in studying first the solution $u^{\varepsilon}$ to the non-linear differential Dirichlet problem 
\begin{equation}\label{p13.0}
  \begin{cases}
  qu^{\varepsilon}-\Gamma  u^{\varepsilon}+ \psi_{\varepsilon}(|\deri^{1} u^{\varepsilon}|^{2})=h, & \text{in}\ B_{R}(0),\\
   u^{\varepsilon}=0,   &\text{on}\ \partial B_{R}(0);
  \end{cases}
  \end{equation}
where the \textit{penalizing function} $\psi_{\varepsilon}:\R\longrightarrow\R$, with $\varepsilon\in(0,1)$, is defined by 
\begin{equation}\label{p12.2}
\psi_{\varepsilon}(r)\eqdef\psi\biggl(\frac{r-1}{\varepsilon}\biggr),\  \text{for all}\ r\in\R,  
\end{equation}
with $\psi\in\hol^{\infty}(\R)$ such that
\begin{equation}\label{p12.1}
  \begin{cases}
  \psi(r)=0,& \text{for all}\ r\leq0,\\
  \psi(r)>0,& \text{for all}\ r>0,\\
  \psi(r)=r-1,& \text{for all}\ r\geq2,\\
  \psi'(r)\geq0,\ \psi''(r)\geq0, & \text{for all}\ r\in\R.
  \end{cases}
\end{equation}
The method used in (\ref{p13.0}) is usually called penalty method and was introduced by L. C. Evans  to establish existence and regularity of solutions   to second order  elliptic equations  with gradient constraints \cite{evans}. This method has  also been  used in other works like \cite{ishii, soner, hynd, hynd2}.  

To study the HJB equation (\ref{p1}), in the case that the operator $\Gamma$ is given by (\ref{p6}), the literature suggests us that we need first to guarantee  the existence and regularity of the solution to the Dirichlet problem (\ref{p13.0}), with  $\Gamma$ as in (\ref{p6}). Once this is done, we need to  establish  uniform estimates of the solutions to the non-linear integro-differential Dirichlet (NIDD) problem (\ref{p13.0}) that allow us to pass to  the limit as $\varepsilon\rightarrow0$,  in a weak sense in (\ref{p13.0}), which leads to the existence and regularity of the solution to the HJB equation (\ref{p1}). 

The following hypotheses will be assumed throughout the paper.

\noindent\textbf{Hypotheses}
{\it
\begin{enumerate}[(H1)]
\item\label{h1}The function $h\in\hol^{2}(\overline{B_{R}(0)})$ is non-negative and $||h||_{\hol^{2}(\overline{B_{R}(0)})}\leq C_{0}$, for some constant $C_{0}>0$.

\item\label{h2}The L\'evy measure $\nu$ satisfies $\nu_{0}\eqdef\nu(\R^{*})<\infty$ and 
\begin{equation*}
\nu_{1}\eqdef\int_{\R^{*}}|z|\nu(\der z)<\infty.
\end{equation*}
In addition, we assume that  $\nu$   has a density $\kappa$ with respect to the Lebesgue measure $\der z$, i.e. $\nu(\der z)=\kappa(z)\der z$, such that $\kappa\in\hol^{0,\alpha}(\R^{*})$, for some $\alpha\in(0,1)$ fixed.

\item \label{h3} There exist real numbers $0<\theta\leq\Theta$  such that the coefficients of the differential part of  $\Gamma$ satisfy  
\begin{equation*}
\theta|\zeta|^{2}\leq\langle\sigma \zeta,\zeta\rangle\leq \Theta|\zeta|^{2},\ \text{for all}\ \zeta\in \R^{d}.
\end{equation*}

\item\label{h4} The discount parameter $q$ is large enough and such that
\begin{equation*}
2A_{0}\nu(B_{R+\frac{b}{2}}(0))<q+\nu_{0}\defeq q',\ \text{with}\ A_{0}\in (1,2). 
\end{equation*}

\end{enumerate}
}
Let us now make some comments on the hypotheses (H\ref{h1})--(H\ref{h4}). The Hypotheses (H\ref{h1}) and (H\ref{h4}) allow us to ensure the existence, uniqueness and regularity to the non-negative solution $u^{\varepsilon}$ of the  NIDD problem (\ref{p13.0}); see Theorem \ref{solNIDD.1.0} and Propositions \ref{intop4.1}--\ref{upos1}. The main reason of the Hypothesis (H\ref{h2}) is that it is necessary  to guarantee  the existence, uniqueness and regularity of the solution $\mathpzc{u}^{\varepsilon}(\cdot;w)$ to the non-linear Dirichlet problem (\ref{ed3}), when $w\in\hol^{0}(\overline{B_{R}(0)})$; see Subsection \ref{NLDP.0.1}. Defining the map $T_{\varepsilon}$ as in (\ref{mapT.1.0}) and using the Hypothesis (H\ref{h4}), we prove $T_{\varepsilon}$ is a contraction mapping in $(\hol^{0}(\overline{B_{R}(0)}),||\cdot||_{\hol^{0}(\overline{B_{R}(0)})})$; see Lemma \ref{contrac1}.  Then, by the contraction fixed point Theorem; see  \cite[Thm. 5.1 p.74]{gilb}, we verify the existence, uniqueness and regularity  to the solution $u^{\varepsilon}$ of the NIDD problem (\ref{p13.0}). Finally, Hypothesis (H\ref{h3}) is a classical assumption for differential operators called \textit{ellipticity property}, see, e.g. \cite{evans,ishii,lady, gimb,gilb,garroni,mark,hynd,bayr}.

Under the assumptions (H\ref{h1})--(H\ref{h4}), the main results obtained in this do\-cu\-ment are the following. 
\begin{teor}\label{princ1.0.1}
If $d< p<\infty$, there exists a unique non-negative solution $u$ to the HJB equation (\ref{p1}) in the space  $\hol^{0,1}(\overline{B_{R}(0)})\cap\sob^{2,p}_{\loc}(B_{R}(0))$, where the operator $\Gamma$ is as in (\ref{p6}).
\end{teor}

\begin{teor}\label{princ1.0}
For each $\varepsilon\in(0,1)$, there exists a unique non-negative solution $u^{\varepsilon}$ to the NIDD problem (\ref{p13.0}) in the space $\hol^{3,\alpha}(\overline{B_{R}(0)})$, where the operator $\Gamma$ is as in (\ref{p6}).
\end{teor}

Previous to this work, the  equations (\ref{p1}) and (\ref{p13.0}) have mainly  been studied in the case that $\Gamma$ is an elliptic differential operator; see, e.g.  \cite{evans, ishii, soner,  menaldi, kruk, hynd2}. The closest to our work is the the paper of Menaldi and Robin \cite{menal2}. 

The solution $u$ to the HJB equation (\ref{p1}), obtained here, is in a strong sense which should be contrasted with recent results in the topic, where the solutions are established in the viscosity sense; see \cite{mark,bayr}. Although the NIDD problem (\ref{p13.0}) is a tool to guarantee the existence of the HJB equation (\ref{p1}), this turns out to be a problem of interest in itself  because it is also related with optimal stochastic control problems where the state process is a controlled $d$-dimensional L\'evy process as in (\ref{p8.1}). The optimal stochastic control problems related to the NIDD equation (\ref{p13.0})  will be analyzed in Subsection \ref{prob1}.

We will establish Theorems \ref{princ1.0.1} and \ref{princ1.0}, by probabilistic, integro-differential and PDE classical methods, inspired by Evans \cite{evans}, Lenhart \cite{lenh}, Gimbert and Lions \cite{gimb}, Soner and Shreve \cite{soner}, Garroni and Menaldi \cite{garroni} and Hynd \cite{hynd2}. Since $\nu(\R^{*})<\infty$, we have that the HJB equation (\ref{p1}) can be written as
\begin{equation}\label{p.1.0}	
\begin{cases}
  \max\{q'u- \Gamma^{'} u-h,|\deri^{1} u|^{2}-1\}=0, & \text{a.e. in}\ B_{R}(0),\\
  u=0,&\text{on}\ \partial B_{R}(0),
\end{cases}  
\end{equation}
where
\begin{equation}\label{p.1.0.0}
\begin{cases}
q'\eqdef q+\nu(\R^{*})= q+\nu_{0},\\
\Gamma^{'}u(x)\eqdef\frac{1}{2}\tr(\sigma\deri^{2} u(x))+\langle\deri^{1} u(x),\widetilde{\gamma}\rangle+\int_{\R^{*}}E(u)(x+z)\nu(\der z)\defeq \dif^{'}u(x)+\inted E(u)(x).
\end{cases}
\end{equation}
The differential and integral part of $\Gamma^{'}$ are  denoted by $\dif^{'}$ and $\inted$, respectively. Furthermore, consider same way than (\ref{p.1.0}), using (\ref{p.1.0.0})  we see that  the NIDD problem (\ref{p13.0}) can be written as
\begin{equation}\label{p.13.0}
  \begin{cases}
  q'u^{\varepsilon}-\Gamma^{'}  u^{\varepsilon}+ \psi_{\varepsilon}(|\deri^{1} u^{\varepsilon}|^{2})=h, & \text{in}\ B_{R}(0),\\
   u^{\varepsilon}=0,   &\text{on}\ \partial B_{R}(0),
  \end{cases}
  \end{equation}
Then, Theorems \ref{princ1.0.1} and \ref{princ1.0} are equivalent to proving that:  
{\it
\begin{enumerate}[(i)]
\item if $d< p<\infty$, there exists a unique non-negative solution $u$ to the HJB equation (\ref{p.1.0}) in the space  $\hol^{0,1}(\overline{B_{R}(0)})\cap\sob^{2,p}_{\loc}(B_{R}(0))$;

\item for each $\varepsilon\in(0,1)$, there exists a unique non-negative solution $u^{\varepsilon}$ to the NIDD problem (\ref{p.13.0}) in the space $\hol^{3,\alpha}(\overline{B_{R}(0)})$.
\end{enumerate}}
 
Finally let us comment that although in this work it is not established  the existence, regularity and uniqueness of the solution to the HJB equation (\ref{p1}) on the whole space, we will give a description of  the existent  relationship between this HJB equation and a singular stochastic control problem, where the controlled process $Y$ is a $d$-dimensional L\'evy process as in (\ref{p8.1}); see Subsection \ref{prob1} below. The study of  the HJB equation (\ref{p1}) on the whole space, is a work in progress by the author. 

In the following subsection, we shall explain the relationship between the equation (\ref{p1}) on the whole space and the singular stochastic control problem given by \eqref{p3}, and also the relationship between the equation (\ref{p13.0}) and the  stochastic control problem given by (\ref{conv3}); see Lemmas  \ref{verif} and \ref{convexu1.0} below. 

\subsection{Probabilistic interpretation}\label{prob1}
Through out this document, we will work on a filtered and complete probability space $(\Omega,\,\mathcal{F},\,\F=\{\mathcal{F}_{t}\}_{t\geq0},\,\Pro)$. The filtration $\F=\{\mathcal{F}_{t}\}_{t\geq0}$ is the one  generated by the $d$-dimensional L\'evy process   $Y=\{Y_{t}:t\geq0\}$, which is given by
\begin{align}\label{p8.1}
  Y_{t}=W_{t} +\widetilde{\gamma} t+\int_{[0,t]}\int_{\R^{*}}z\,\vartheta(\der s\times \der z),\ \text{for all}\ t\geq0,
\end{align}
where  $W=\{W_{t}:t\geq0\}$ is a $d$-dimensional Brownian motion with Gaussian covariance  matrix $\sigma$, $ \widetilde{\gamma}\in\R^{d}$ as in (\ref{p.1.0.1}), and $\vartheta$ is a Poisson random measure in $[0,\infty)\times\R^{*}$ equipped of the $\sigma$-algebra $\mathcal{B}$   generated by $\mathcal{B}[0,\infty)\times\mathcal{B}(\R^{*})$, with an intensity measure $\der t\times\nu(\der z)$.   The last part on the right side in (\ref{p8.1}) is a compound Poisson process with rate $\nu(\R^{*})$ and the distribution of its jumps is $\nu(\R^{*})^{-1}\nu(\der z)$. We assume furthermore that the filtration $\F$ is completed with the null sets of $\Pro$.

By the L\'evy-Khintchine formula \cite[p. 37]{sato} it is well known  that the L\'evy process $Y$ is determined by a triplet  $(\widetilde{\gamma},\sigma,\nu)$, where $\widetilde{\gamma}\in\R^{d}$ as in (\ref{p.1.0.1}), $\sigma$ is a positive definite matrix  of size $d\times d$ that satisfies (H\ref{h3}) and $\nu$ is a measure on $\R^{*}$ that satisfies (H\ref{h2}). In the present case  the characteristic exponent has the following form
\begin{align*}
  \Psi(\lambda)&=-\log(\E(\expo^{\I\langle\lambda,Y\rangle}))=-\I\langle\widetilde{\gamma},\lambda\rangle+\frac{1}{2}\langle\lambda\sigma,\lambda\rangle-\nu(\R^{*})\int_{\R^{*}}(\expo^{\I\langle\lambda,z\rangle}-1)\frac{ \nu(\der z)}{\nu(\R^{*})},
\end{align*}
for all $\lambda\in\R^{d}$, and we recall that  its infinitesimal generator is given by  (\ref{p1.0.5}). The \textit{state process} $X=\{X_{t}:t\geq0\}$ is defined as
  \begin{equation}\label{p3.0}
	X_{t}\eqdef x+Y_{t}+Z_{t},\ \text{for all}\ t\geq0,
\end{equation}
where $x\in\R^{d}$ is the initial condition, $Y=\{Y_{t}:t\geq0\}$ is a $d$-dimensional L\'evy process as in (\ref{p8.1}), and $Z=\{Z_{t}:t\geq0\}$ is a \textit{control process}. 

\subsubsection*{Probabilistic interpretation of the HJB equation  on the whole space}

In addition to the Hypotheses (H\ref{h1})--(H\ref{h4}), we need here to assume  others hypotheses. The reason for this, is to establish the existent relationship between the HJB equation (\ref{p1}) on the whole space and a singular stochastic control problem.  Assume that the  L\'evy measure of the process $Y$, $\nu$,  satisfies  
\begin{equation}\label{aditH1}
\int_{\R^{*}}\nu(\der z)<\infty\ \text{and}\ \int_{\R^{*}}\max\{|z|,|z|^{2}\}\nu(\der z)<\infty, 
\end{equation}
and here the control process $Z$ is given by $Z_{t}\eqdef\int_{[0,t]}N_{s}\,\der\xi_{s}$, for $t\geq0$, where $(N,\xi)=\{(N_{t},\xi_{t}):t\geq0\}$ is $\F$-adapted, $|N_{t}|=1\ \Pro\text{-a.s.}$, and $\xi$ is a nondecreasing, left-continuous process with $\xi_{0}=0$ $\Pro$-a.s.. Then,  the state process $X=\{X_{t}:t\geq0\}$ given in (\ref{p3.0}) takes the following form
  \begin{equation}\label{p3}
	X_{t}= x+Y_{t}+\int_{[0,t]}N_{s}\,\der\xi_{s},\ \text{for all}\ t\geq0.
\end{equation}
The process $N$ provides the direction and $\xi$  the intensity of the push applied to the state process $X$. Note that the jumps of the state process $X$ are inherited from $Y$ and $\xi$, and we assume that these processes do not jump at the same time $t$, i.e.
\begin{equation}\label{salto}
  \Delta X_{t}=X_{t}-X_{t-}=\Delta Y_{t}\uno_{\{\Delta Y_{t}\neq0,\,\Delta\xi_{t}=0\}}+N_{t}\Delta \xi_{t}\uno_{\{\Delta\xi_{t}\neq0,\,\Delta Y_{t}=0\}},
\end{equation}
for all $t\geq0$. For  $q>0$  and a control process $(N,\xi)$, the corresponding \textit{cost function} is defined as
\begin{equation*}
	V_{(N,\xi)}(x)=\E_{x}\biggl(\int_{[0,\infty)}\expo^{-qt}(h(X_{t})\,\der t+\der\xi_{t})\biggr),\ \text{for all}\ x\in\R^{d},
\end{equation*}
where $h\in \hol^{2,1}_{\loc}(\R^{d})$ is a strictly convex function satisfying for some positive constants $C_{0}$ and $c_{0}$,
  \begin{equation}\label{convh1}
  	\begin{cases}
    		0=h(0)\leq h(x)\leq C_{0}(1+|x|^{2}),\\
    		|\deri^{1} h(x)|\leq C_{0}(1+h(x)),\\
    		c_{0}|y|^{2}\leq \langle\deri^{2}h(x)y, y\rangle\leq C_{0}|y|^{2}(1+h(x)),
	\end{cases}
\end{equation}  
for all $x,y\in\R^{d}$. From (\ref{aditH1}) and (\ref{convh1}), we have $\E(h(Y_{t}))<\infty$. Then, the \textit{value function} is given by
\begin{equation}\label{p1.0.3}
	V(x)=\inf_{(N,\xi)}V_{(N,\xi)}(x),\ \text{for}\ x\in\R^{d}.
\end {equation}  
A heuristic derivation from dynamic programming principle; see \cite[Ch. VIII]{flem}, shows that the value function $V$ is related to the HJB equation
\begin{equation}\label{p1.0.4}
\max\{qu-\Gamma u-h,|\deri^{1} u|^{2}-1\}=0,
\end{equation}
where 
\begin{equation}\label{p1.0.5.0}
\Gamma u(x)= \frac{1}{2}\tr(\sigma\deri^{2} u(x))+\langle\deri^{1} u(x),\widetilde{\gamma}\rangle+\int_{\R^{*}}(u(x+z)-u(x))\nu(\der z).
\end{equation}
The relationship between the value function (\ref{p1.0.3}) and the HJB equation (\ref{p1.0.4}) is described in the following lemma, whose proof is in the appendix. 

\begin{lema}\label{verif}
Suppose that (\ref{aditH1}) and (\ref{convh1}) hold true. If $u$ is a convex function in $\hol^{2}(\R^{d})$, which is a solution of the HJB equation (\ref{p1.0.4}), then
\begin{enumerate}[(i)]
\item $u(x)\leq V(x)$, for each $x\in\R^{d}$;
\item given the initial condition  $X^{*}_{0}=x,\ x\in\R^{d}$, suppose that there exists a control process $(N^{*},\xi^{*})$ such that $V_{(N^{*},\xi^{*})}(x)<\infty$ and the state process $X^{*}$ satisfies
\begin{equation}\label{demo10.1}
  \begin{cases}
  qu(X^{*}_{t-})-\Gamma u(X^{*}_{t-})-h(X^{*}_{t-})=0, &  \\
  \int_{[0,t]}\uno_{\{N^{*}_{s}=-\deri^{1} u(X^{*}_{s-})\}}\,\der\xi^{*}_{s}=\xi^{*}_{t},\\
  (u(X^{*}_{t-})-u(X^{*}_{t+}))\uno_{\{\Delta\xi^{*}_{t}\neq0,\,\Delta Y_{t}=0\}}=\xi^{*}_{t+}-\xi^{*}_{t}, & 
  \end{cases}
\end{equation}		
for all $t\in[0,\infty)$ a.s., with $\Gamma$ as in (\ref{p1.0.5.0}). Then, $u(x)=V(x)=V_{(N^{*},\xi^{*})}(x)$,
i.e. $(N^{*},\xi^{*})$ is optimal at $x$.
\end{enumerate}
\end{lema}

\begin{rem}
It is important to clarify that to verify the first part of Lemma \ref{verif}, is necessary to show the existence, regularity in $\hol^{2}(\R^{d})$ and convexity of the solution to the HJB equation in \eqref{p1.0.4}. At this stage we have only been able to verify these properties on bounded domains. Establishing these properties in the whole space, is a technical difficult task, and it is the topic of a work in progress by the author. For the second part of Lemma \ref{verif}, one needs to guarantee the existence of a stochastic process that satisfies the conditions given in \eqref{demo10.1}, and hence we obtain the equality between the solution to the HJB equation \eqref{p1.0.4} and the value function defined in  \eqref{p1.0.3}.   
\end{rem}

\subsubsection*{Probabilistic interpretation of the NIDD problem}

Now, we take the control process $Z$ as $Z_{t}\eqdef-\varrho_{t}$, for $t\geq0$, where $\varrho=\{\varrho_{t}:t\geq0\}$ is any $d$-dimensional, absolutely continuous,  $\F$-adapted process, satisfying $\varrho_{0}=0$ almost surely. Then, the state process $X$ given in (\ref{p3.0}), takes the following form
\begin{align}\label{conv3}
  X_{t}= x+Y_{t}-\varrho_{t},\ \text{for all}\ t\geq0,
\end{align}
where the initial state $x$ belongs to $B_{R}(0)$ and  $Y=\{Y_{t}:t\geq0\}$, is a $d$-dimensional L\'evy process as in(\ref{p8.1}). We define the convex function  $g_{\varepsilon}:\R^{d}\longrightarrow\R$ and its Legendre transform $l_{\varepsilon}:\R^{d}\longrightarrow\R$ by
\begin{equation}\label{conv1}
\begin{cases}
g_{\varepsilon}(\zeta)\eqdef\psi_{\varepsilon}(|\zeta|^{2}), \\
l_{\varepsilon}(\eta)\eqdef\sup_{\zeta}\{\langle \eta, \zeta \rangle-g_{\varepsilon}(\zeta)\};
\end{cases}
\end{equation}
where $\psi_{\varepsilon}$ is given in (\ref{p12.2}). Observe that the Legendre transform $l_{\varepsilon}$ is a non-negative function. The \textit{cost function} corresponding  to $\varrho$ is given by
\begin{equation*}
  V_{\varrho}^{\varepsilon}(x)\eqdef\E_{x}\biggl(\int_{0}^{\tau_{B_{R}(0)}}\expo^{-qs}(h(X_{s})+l_{\varepsilon}(\dot{\varrho}_{s}))\,\der s\biggr),
\end{equation*}   
for all $x\in B_{R}(0)$, with $\tau_{B_{R}(0)}\eqdef\inf\{t\geq0:X_{t}\notin B_{R}(0)\}$ and $\dot{\varrho}_{t}\eqdef\frac{\der\varrho_{t}}{\der t}$. Then, the \textit{value function} is defined by
\begin{equation}\label{V.1}
  V^{\varepsilon}(x)\eqdef\inf_{\varrho}V_{\varrho}^{\varepsilon}(x).
\end{equation}	

Note that the functions in (\ref{conv1}) satisfies the following property
\begin{equation}\label{conv2.0.0}
g_{\varepsilon}(\zeta)=2\psi'_{\varepsilon}(|\zeta|^{2})|\zeta|^{2}-l_{\varepsilon}(2\psi'_{\varepsilon}(|\zeta|^{2})\zeta),\ \text{for}\ \zeta\in\R^{d}.
\end{equation}
Since $g_{\varepsilon}$ is differentiable, it follows that $g_{\varepsilon}(\zeta)=\sup_{\eta}\{\langle \zeta,\eta\rangle-l_{\varepsilon}(\eta)\}$. Then, the NIDD problem (\ref{p13.0}) can be written as
\begin{equation}\label{p13.0.1.0}
  \begin{cases}
  qu^{\varepsilon}(x)-\Gamma  u^{\varepsilon}(x)+ \sup_{\eta}\{\langle \deri^{1}u^{\varepsilon}(x),\eta\rangle-l_{\varepsilon}(\eta)\}=h(x), & \text{in}\ B_{R}(0),\\
   u^{\varepsilon}(x)=0,   &\text{on}\ \partial B_{R}(0),
  \end{cases}
\end{equation}
where $\Gamma$ is as in (\ref{p1.0.5}).  
Using that $u^{\varepsilon}\in\hol^{3,\alpha}(\overline{B_{R}(0)})$ is the solution to the NIDD problem (\ref{p13.0.1.0}) (see Theorem \ref{princ1.0}),  we obtain  the following lemma, whose proof is in the appendix.
\begin{lema}\label{convexu1.0}
The solution $u^{\varepsilon}$ to the NIDD problem (\ref{p13.0.1.0}) agrees with $V^{\varepsilon}$ in $\overline{B_{R}(0)}$.
\end{lema}
\begin{rem}
Since $u^{\varepsilon}=V^{\varepsilon}$ in $\overline{B_{R}(0)}$ and (\ref{p13.0.1.0}) is equivalent to (\ref{p13.0}), we can deduce from here that the solution  $u^{\varepsilon}$ to the NIDD problem (\ref{p13.0}) is unique.
\end{rem}

Finally, we introduce the notation and basic definitions of some spaces that are used in this paper. Let $\set\subseteq\R^{d}$ be an open set, $\alpha\in[0,1]$ and $m\in\{0,\dots,k\}$, with $k\geq0$ an integer. The set $\hol^{k}(\set)$ consists of real valued functions on $\set$ that are $k$-fold differentiable. We define $\hol^{\infty}(\set)=\bigcap_{k=0}^{\infty}\hol^{k}(\set)$. The sets $\hol^{k}_{\comp}(\set)$  and $\hol^{\infty}_{\comp}(\set)$ consist  of functions in $\hol^{k}(\set)$ and $\hol^{\infty}(\set)$, whose support is compact and contained in $\set$, respectively. The set $\hol^{k}(\overline{\set})$ is defined as the set of real valued functions such that $\partial^{a}f$ is  bounded and uniformly continuous  on $\set$, for every $a\in\ind_{m}$, with $\ind_{m}$ the set of all multi-indices of order $m$. This space is equipped with the following norm $||f||_{\hol^{k}(\overline{\set})}=\sum_{m=0}^{k}\sum_{a\in\ind_{m}}\sup_{x\in\set}\{|\partial^{a}f(x)|\}$, where $ \sum_{a\in\ind_{m}}$ denotes summation over all possible $m$-fold derivatives of $f$.  For each $D\subseteq\R^{d}$ and $f:D\longrightarrow\R$, the operator $[\cdot]_{\hol^{0,\alpha}(D)}$ is given by $[f]_{\hol^{0,\alpha}(D)}\eqdef\sup_{x,y\in D,\, x\neq y}\Bigl\{\frac{|f(x)-f(y)|}{|x-y|^{\alpha}}\Bigr\}$. We define $\hol^{k,\alpha}_{\loc}(\set)$ as the set of functions in $\hol^{k}(\set)$ such that $[\partial^{a}f]_{\hol^{0,\alpha}(K)}<\infty$, for every compact set $K\subset\set$ and every $a\in\ind_{m}$. The set $\hol^{k,\alpha}(\overline{\set})$ denotes the set of all functions in $\hol^{k}(\overline{\set})$ such that $[\partial^{a}f]_{\hol^{0,\alpha}(\set)}<\infty$, for every $a\in\mathcal{D}_{m}$.  This set is equipped with the following norm
\begin{equation}\label{sp7}
||f||_{\hol^{k,\alpha}(\overline{\set})}=\sum_{m=0}^{k}\sum_{a\in\ind_{m}}\Bigr(\sup_{x\in \set}|\partial^{a}f(x)|+[\partial^{a}f]_{\hol^{0,\alpha}(\set)}\Bigl).
\end{equation}
We understand $\hol^{k,\alpha}(\R^{d})$ as $\hol^{k,\alpha}(\overline{\R^{d}})$, in the sense that  $[\partial^{a}f]_{\hol^{0,\alpha}(\R^{d})}<\infty$, for every $a\in\ind_{m}$. As usual, $\Lp^{p}(\set)$ with $1\leq p<\infty$, denotes the class of real valued functions on $\set$ with finite norm $||f||^{p}_{\Lp^{p}(\set)}\eqdef\int_{\set}|f|^{p}\der x<\infty$, where $\der x$ denotes the Lebesgue measure. Also, let  $\Lp^{p}_{\loc}(\set)$ consist of functions whose $\Lp^{p}$-norm is finite on any compact subset of $\set$. Define the Sobolev space $\sob^{k,p}(\set)$ as the class of functions  $f\in\Lp^{p}(\set)$ with weak or distributional partial derivatives $\partial^{a}f$, see \cite[p. 22]{adams}, and with finite norm
\begin{equation}\label{normw}
	||f||^{p}_{\sob^{k,p}(\set)}=\sum_{m=0}^{k}\sum_{a\in\ind_{m}}||\partial^{a}f||^{p}_{\Lp^{p}(\set)},\ \text{for all}\ f\in\sob^{k,p}(\set).
\end{equation} 
The space $\sob^{k,p}_{\loc}(\set)$ consists of functions whose $\sob^{k,p}$-norm is finite on any compact subset of $\set$. When $p=\infty$, the Sobolev and Lipschitz spaces are related. In particular, $\sob^{k,\infty}_{\loc}(\set)=\hol^{k-1,1}_{\loc}(\set)$ for an arbitrary subset $\set\subseteq\R^{d}$, and $\sob^{k,\infty}(\set)=\hol^{k-1,1}(\overline{\set})$ for a sufficiently smooth domain $\set$, when it is Lipschitz.  

The rest of this document is organized as follows. Section 2 is devoted to the study some properties of the extension operator $E$. First, we  recall an extension theorem for H\"older spaces (Theorem \ref{exthol1}), whose proof can be found in \cite[p. 353]{stein}. Then, Theorem \ref{exthol1} gives a continuous linear operator $E:\hol^{k,\alpha}(\overline{B_{R}(0)})\longrightarrow\hol^{k,\alpha}_{\comp}(\R^{d})$ with $k\geq0$ and $\alpha\in[0,1]$, which is used to verify that $\inted E(w)$ is well defined when $w\in\hol^{k,\alpha}(\overline{B_{R}(0)})$.  We also discuss some properties of $\inted E(w)$, whenever $w\in\hol^{k}(\overline{B_{R}(0)})$. In Section 3 we disclose  the  existence, uniqueness and regularity to the non-linear Dirichlet problems (\ref{p.13.0}) and (\ref{ed3}); the former with an integro-differential operator,  and the latter with a differential operator. We also discuss some properties of these solutions. In Section 4 we establish the existence, uniqueness and regularity of the HJB equation (\ref{p.1.0}), which is equivalent to (\ref{p1}). 

\section{Extension theorem and properties of the integral operator}\label{Ch2}

In the first part of this section, we shall describe the extension operator $E$ that appears in (\ref{p13.1.0}). Since the construction of this operator is long and the arguments used in its study are not required in the rest of the paper, we remit  the reader to \cite{stein, csa, moreno} for details. At the end of the section, we show useful properties of  $\inted E(w)(x)=\int_{R^{*}}E(w)(x+z)\nu(\der z)$. The proofs of the results  of this section are in \cite[Ch. 2]{moreno}.
 \begin{teor}[Extension theorem for H\"older spaces]\label{exthol1}
For any positive integer $k$ and any $\alpha\in[0,1]$, there exists a continuous linear extension operator $E:\hol^{k,\alpha}(\overline{B_{R}(0)})\longrightarrow\hol^{k,\alpha}_{\comp}(\R^{d})$, that satisfies (\ref{p13.1.0}).
\end{teor} 

\begin{prop}[\cite{moreno}, Prop. 2.15, p. 40]\label{cotEu1}
If $ w\in\hol^{0}(\overline{B_{R}(0)})$, then there exists a constant $A_{0}\in(1,2)$ such that $|E(w)(x)|\leq 2A_{0}||w||_{\hol^{0}(\overline{B_{R}(0)})}$, for all $x\in\R^{d}$.
\end{prop}

\begin{lema}[\cite{moreno}, Lemma 2.16, p. 41]\label{cotaext2}
If $w\in\hol^{1}(\overline{B_{R}(0)})$,  there exists a constant $C_{1}=C_{1}(k,d)>0$ such that for each $x\in\mathcal{B}'\eqdef B_{R+\frac{b}{2}}(0)\setminus \overline{B_{R}(0)}$, $|\partial_{i}E(w)(x)|< C_{1}||w||_{\hol^{1}(\overline{B_{R}(0)})},$ for all $i\in\{1,\dots, d\}$.
\end{lema}

\begin{lema}[\cite{moreno}, Lemma 2.17, p. 43]\label{lemalip.1.0}
\mbox{}
\begin{enumerate}[(i)]
\item If $w\in\hol^{0}(\overline{B_{R}(0)})$, then $|\inted E (w)(x)|\leq 2A_{0}\nu_{0}||w||_{\hol^{0}(\overline{B_{R}(0)})}$, for all $x\in \R^{d}$, where $\nu_{0}$, $A_{0}$  are as in (H\ref{h2}) and Proposition \ref{cotEu1}, respectively. 

\item If $w\in\hol^{0}(\overline{B_{R}(0)})$, then $\inted E(w)\in\hol^{0,\alpha}(\R^{d})$.

\item If $w\in\hol^{1}(\overline{B_{R}(0)})$, then $\partial_{i}\inted E(w)\in\hol^{0,\alpha}(\R^{d})$ and $\partial_{i}\inted E(w)=\inted \partial_{i}E(w)$, for each $i\in\{1,\dots,d\}$.
\end{enumerate}
\end{lema}
The following corollary is an immediate consequence  of the  previous lemma. Recall that $\ind_{m}$, with $0\leq m\leq k$, is the set of all multi-indices of order $m$.
\begin{coro}\label{lemalip}
Let $k\geq0$ be  an integer. If $w\in\hol^{k}(\overline{B_{R}(0)})$, then $\inted E(w)\in\hol^{k,\alpha}(\R^{d})$.
\end{coro}

The following two lemmas describe the behavior of $\inte E(w)$ and $\inte\partial_{i}E(w)$, when $x+z\in \overline{B_{R}(0)}^{\,\comp}$ and $x\in B_{R}(0)$, but we first choose  an integer $N\geq1$   large enough,  $x_{\kappa}\in\partial B_{R}(0)$ and $b_{\kappa}>0$ small enough,  with $\kappa\in\{1,\dots,N\}$, such that $\partial B_{R}(0)\subseteq \bigcup_{\kappa=1}^{N}B_{b_{\kappa}}(x_{\kappa})$. Taking $0<b< \min_{\kappa\in\{1,\dots,N\}}\bigr\{\frac{1}{2^{N}},b_{\kappa}\bigl\}$  such that $\partial B_{R}(0)\subseteq \bigcup_{\kappa=1}^{N}B_{b_{\kappa}-\frac{b}{2}}(x_{\kappa})$, we assume that
\begin{equation*}
\begin{cases}
x_{\kappa'}\notin B_{b_{\kappa}}(x_{\kappa}),\ \text{with}\ \kappa,\kappa'\in\{1,\dots,N\}\ \text{and}\ \kappa\neq \kappa',\\
B_{b_{N}-\frac{b}{2}}(x_{N})\cap B_{b_{1}-\frac{b}{2}}(x_{1})\neq\varnothing,\\
B_{b_{\kappa}-\frac{b}{2}}(x_{\kappa})\cap B_{b_{\kappa+1}-\frac{b}{2}}(x_{\kappa+1})\neq\varnothing,\ \text{for any}\ \kappa\in\{1,\dots, N-1\}.
\end{cases}
\end{equation*} 
The previous assumption holds, since $\partial B_{R}(0)$ is a compact set.
\begin{lema}[\cite{moreno}, Lemma 2.19, p. 44]\label{cotaintaf1}
If $w\in\hol^{0}(\overline{B_{R}(0)})$, then 
\begin{equation*}
\biggl|\int_{\{|x+z|> R\}}E(w)(x+z)\nu(\der z)\biggr|\leq2A_{0}||w||_{\hol^{0}(\mathcal{B}'_{1})}\nu(\mathcal{B}'),
\end{equation*}
for all $x\in B_{R}(0)$, where $A_{0}$ is a constant given in Proposition \ref{cotEu1}, $\mathcal{B}'= B_{R+\frac{b}{2}}(0)\setminus\overline{B_{R}(0)}$ and $\mathcal{B}'_{1}\eqdef\overline{B_{R}(0)}\cap\bigcup_{\kappa=1}^{N}\overline{B_{b_{\kappa}-\frac{b}{4}}(x_{\kappa})}$.
\end{lema}

\begin{lema}[\cite{moreno}, Lemma 2.20, p. 45]\label{intcotadif1.1}
If  $w\in\hol^{1}(\overline{B_{R}(0)})$, then 
\begin{align*}
\biggl|\int_{\{|x+z|> R\}}\partial_{i}E(w)(x+z)\nu(\der z)\biggr|&\leq C_{1}||w||_{\hol^{1}(\mathcal{B}'_{1})}\nu(\mathcal{B}'),
\end{align*}
for all $x\in B_{R}(0)$, where $C_{1}$  is a constants given in Lemma \ref{cotaext2}, respectively, $\mathcal{B}'= B_{R+\frac{b}{2}}(0)\setminus\overline{B_{R}(0)}$ and $\mathcal{B}'_{1}=\overline{B_{R}(0)}\cap \bigcup_{\kappa=1}^{N}\overline{B_{b_{\kappa}-\frac{b}{4}}(x_{\kappa})}$.
\end{lema}

\section[Non-linear Dirichlet problems]{Non-linear Dirichlet problems}
In this section, we are interested in establishing the existence, uniqueness and regularity of the solution to the non-linear integro-differential Dirichlet (NIDD)  problem (\ref{p.13.0}). The arguments used here,  are based  in the contraction fixed point  Theorem; see \cite[Thm. 5.1 p.74]{gilb}. 

\subsection{Non-linear Dirichlet problem with an elliptic differential operator}\label{NLDP.0.1}

For each $w\in\hol^{0}(\overline{B_{R}(0)})$, define $\widetilde{h}(\cdot;w):\overline{B_{R}(0)}\longrightarrow\R$ as  
\begin{equation*}
\widetilde{h}(x;w)\eqdef h(x)+\inted E(w)(x),\ \text{for}\ x\in\overline{B_{R}(0)}.
\end{equation*}
Since $h\in\hol^{2}(\overline{B_{R}(0)})$ and $\inted E(w)\in\hol^{0,\alpha}(\R^{d})$, whenever $w\in\hol^{0}(\overline{B_{R}(0)})$; this is due to Hypothesis (H\ref{h1}) and Lemma \ref{lemalip.1.0}(ii) respectively, we have that $\widetilde{h}(\cdot;w)\in\hol^{0,\alpha}(\overline{B_{R}(0)})$. Then, from \cite[Thm 15.10 p. 380]{gilb} and taking $q'$ and $\dif^{'}$ as in (\ref{p.1.0.0}), we have that the non-linear  Dirichlet problem
\begin{equation}\label{ed3}
\begin{cases}
  q'\mathpzc{u}^{\varepsilon}(\cdot;w)-\dif^{'} \mathpzc{u}^{\varepsilon}(\cdot;w)+ \psi_{\varepsilon}(|\deri^{1} \mathpzc{u}^{\varepsilon}(\cdot;w)|^{2})=\widetilde{h}(\cdot;w),& \text{in}\ B_{R}(0),\\
  \mathpzc{u}^{\varepsilon}(\cdot;w)=0, &\text{on}\ \partial B_{R}(0),
\end{cases}
\end{equation} 
has a solution $\mathpzc{u}^{\varepsilon}(\cdot;w)\in\hol^{2,\alpha}(\overline{B_{R}(0)})$. Recall that $\psi_{\varepsilon}$ is defined in (\ref{p12.2}). The uniqueness of $\mathpzc{u}^{\varepsilon}$ is obtained in the following result.
\begin{lema}\label{uniqu1}
The non-linear Dirichlet problem (\ref{ed3}) has a unique solution.
\end{lema}

\begin{proof}
Let $w\in\hol^{0}(\overline{B_{R}(0)})$ and $\varepsilon\in(0,1)$ be fixed. If $\mathpzc{u}^{\varepsilon}_{1}(\cdot;w)$ and $\mathpzc{u}^{\varepsilon}_{2}(\cdot;w)$ are two solutions to the non-linear Dirichlet problem (\ref{ed3}), we  define  $f(\cdot)\eqdef \mathpzc{u}^{\varepsilon}_{1}(\cdot;w)-\mathpzc{u}^{\varepsilon}_{2}(\cdot;w)$ in $\overline{B_{R}(0)}$, which is in $\hol^{2,\alpha}(\overline{B_{R}(0)})$ and 
\begin{equation}\label{NLDir1}
\begin{cases}
q'f-\dif^{'} f+ \psi_{\varepsilon}(|\deri^{1} u^{\varepsilon}_{1}(\cdot;w)|^{2})\\
\hspace{4cm}-\psi_{\varepsilon}(|\deri^{1} u^{\varepsilon}_{2}(\cdot;w)|^{2})=0,& \text{in}\ B_{R}(0),\\
  f=0, &\text{on}\ \partial B_{R}(0).
\end{cases}
\end{equation}
Let $x^{*}\in\overline{B_{R}(0)}$ be the point where $f$ attains its maximum. If $x^{*}\in\partial B_{R}(0)$, from (\ref{NLDir1}), it follows that $f(x)\leq f(x^{*})=0$. Suppose now that $x^{*}\in B_{R}(0)$. Then, we have 
$\deri^{1}f(x^{*})=0$, and $\frac{1}{2}\tr(\sigma\deri^{2}f(x^{*}))\leq 0$,
which imply that 
\begin{equation*}
\psi_{\varepsilon}(|\deri^{1} \mathpzc{u}_{1}(x^{*};w)|^{2})-\psi_{\varepsilon}(|\deri^{1} \mathpzc{u}_{2}(x^{*};w)|^{2})=0,
\end{equation*}
and evaluating $x^{*}$ in (\ref{NLDir1}), it follows $0\geq\frac{1}{2}\tr(\sigma\deri^{2}f(x^{*}))=q'f(x^{*})$, and hence $\mathpzc{u}_{1}^{\varepsilon}(x;w)-\mathpzc{u}_{2}^{\varepsilon}(x;w)\leq f(x^{*})\leq 0$ in $B_{R}(0)$. By symmetry we have also that  $\mathpzc{u}_{2}^{\varepsilon}(\cdot;w)-\mathpzc{u}_{1}^{\varepsilon}(\cdot;w)\leq0$ in $B_{R}(0)$. Therefore $\mathpzc{u}_{1}^{\varepsilon}(\cdot;w)= \mathpzc{u}_{2}^{\varepsilon}(\cdot;w)$, and then, the non-linear Dirichlet problem (\ref{ed3}) has a unique solution.
\end{proof}

Using (\ref{conv1}), we see that the non-linear Dirichlet problem (\ref{ed3}) can be written as \begin{equation}\label{p13.0.1}
  \begin{cases}
  q'\mathpzc{u}^{\varepsilon}(\cdot;w)-\dif^{'} E(\mathpzc{u}^{\varepsilon})(\cdot;w)\\
  \hspace{2cm}+\sup_{\eta}\{\langle \deri^{1}\mathpzc{u}^{\varepsilon}(\cdot;w),\eta\rangle-l_{\varepsilon}(\eta)\}=\widetilde{h}(\cdot;w), & \text{in}\ B_{R}(0),\\
   \mathpzc{u}^{\varepsilon}(\cdot;w)=0,   &\text{on}\ \partial B_{R}(0).
  \end{cases}
\end{equation}
Next we describe the stochastic control problem associated with this equation. Replacing $Y$ by $W$ in (\ref{conv3}), where $W=\{W_{t}:t\geq0\}$ is a $d$-dimensional Brownian motion with Gaussian covariance  matrix $\sigma$ and drift $\widetilde{\gamma}$ as in (\ref{p.1.0.1}), the state process $X=\{X_{t}:t\geq0\}$ takes the following way   
\begin{align*}%\label{conv3}
  X_{t}\eqdef x+W_{t}+\widetilde{\gamma}t-\varrho_{t},\ \text{for all}\ t\geq0,
\end{align*}
where $x\in B_{R}(0)$ and $\varrho$ is any $d$-dimensional, absolutely continuous, $\F$-adapted process, satisfying $\varrho_{0}=0$ $\Pro$-a.s.. The cost function corresponding  of $\varrho$, depending on $w\in\hol^{0}(\overline{B_{R}(0)})$, is given by
\begin{equation*}
  V_{\varrho}^{\varepsilon}(x;w)\eqdef\E_{x}\biggl(\int_{0}^{\tau_{B_{R}(0)}}\expo^{-q's}(\widetilde{h}(X_{s};w)+l_{\varepsilon}(\dot{\varrho}_{s}))\,\der s\biggr),
\end{equation*}   
for all $x\in B_{R}(0)$, with $\tau_{B_{R}(0)}\eqdef\inf\{t\geq0:X_{t}\notin B_{R}(0)\}$ and $\dot{\varrho}_{t}= \frac{\der\varrho_{t}}{\der t}$. The constant $q'>0$ is given in (\ref{p.1.0.0}).  Finally, the value function is defined by
\begin{equation}\label{value.1.0}
  V^{\varepsilon}(x;w)\eqdef\inf_{\varrho}V_{\varrho}^{\varepsilon}(x;w).
\end{equation}	
Recalling that $\mathpzc{u}^{\varepsilon}(\cdot;w)\in\hol^{2,\alpha}(\overline{B_{R}(0)})$, with $w\in\hol^{0}(\overline{B_{R}(0)})$, is the solution to the non-linear Dirichlet problem (\ref{p13.0.1}), we obtain the following result. \begin{lema}\label{convexu1}
The solution $\mathpzc{u}^{\varepsilon}(\cdot;w)$ to the non-linear Dirichlet problem (\ref{p13.0.1}) agrees with $V^{\varepsilon}(\cdot;w)$ in $\overline{B_{R}(0)}$. 
\end{lema}
Since the proof of Lemma \ref{convexu1} is similar to the proof of Lemma \ref{convexu1.0}, we omit it. Defining $T_{\varepsilon}:\hol^{0}(\overline{B_{R}(0)})\longrightarrow\hol^{0}(\overline{B_{R}(0)})$ as 
\begin{equation}\label{mapT.1.0}
T_{\varepsilon}(w)=V^{\varepsilon}(\cdot;w),\ \text{for each}\ w\in\hol^{0}(\overline{B_{R}(0)}), 
\end{equation}
from Lemma \ref{convexu1}, we see that  $T_{\varepsilon}$ is well defined. Now, by Hypothesis (H\ref{h4}) and using the following result; Lemma \ref{contrac1}, we obtain that $T_{\varepsilon}$ is a contraction mapping  in $(\hol^{0}(\overline{B_{R}(0)}),||\cdot||_{\hol^{0}(\overline{B_{R}(0)})})$, and hence, by contraction fixed point Theorem; see \cite[Thm. 5.1 p.74]{gilb}, we have that $T_{\varepsilon}$ has a unique point in $\hol^{0}(\overline{B_{R}(0)})$; see Lemma \ref{contrac2}.
\begin{lema}\label{contrac1}
If $w_{1},w_{2}\in\hol^{0}(\overline{B_{R}(0)})$, then
\begin{equation*}
||V^{\varepsilon}(\cdot;w_{1})-V^{\varepsilon}(\cdot;w_{2})||_{\hol^{0}(\overline{B_{R}(0)})}\leq \frac{2A_{0}\nu(B_{R+\frac{b}{2}}(0))}{q'}||w_{1}-w_{2}||_{\hol^{0}(\overline{B_{R}(0)})}.
\end{equation*}
\end{lema}

\begin{proof}
Let $w_{1},w_{2}\in\hol^{0}(\overline{B_{R}(0)})$. For each $x\in\overline{B_{R}(0)}$, we have
\begin{align}\label{opt1}
V^{\varepsilon}(x;w_{1})&\leq\inf_{\varrho}\biggr\{\sup_{\varrho}\{V^{\varepsilon}_{\varrho}(x;w_{1})-V^{\varepsilon}_{\varrho}(x;w_{2})\}+V^{\varepsilon}_{\varrho}(x;w_{2})\biggr\}\notag\\
&\leq\sup_{\varrho}\{V^{\varepsilon}_{\varrho}(x;w_{1})-V^{\varepsilon}_{\varrho}(x;w_{2})\}+V^{\varepsilon}(x;w_{2}).
\end{align}
Therefore $V^{\varepsilon}(x;w_{1})-V^{\varepsilon}(x;w_{2})\leq\sup_{\varrho}\{V^{\varepsilon}_{\varrho}(x;w_{1})-V^{\varepsilon}_{\varrho}(x;w_{2})\}$. Proceeding of the same way than  (\ref{opt1}),  it yields 
\begin{equation*}
V^{\varepsilon}(x;w_{2})-V^{\varepsilon}(x;w_{1})\leq\sup_{\varrho}(V^{\varepsilon}_{\varrho}(x;w_{2})-V^{\varepsilon}_{\varrho}(x;w_{1})). 
\end{equation*}
Then, using Proposition \ref{cotEu1} and that $\sop[E(w_{2}-w_{1})]\subset B_{R+\frac{b}{2}}(0)$, we conclude that
\begin{align*}%\label{opt2}
|V^{\varepsilon}(x;w_{2})-V^{\varepsilon}(x;w_{1})|&\leq\sup_{\varrho}|V^{\varepsilon}_{\varrho}(x;w_{2})-V^{\varepsilon}_{\varrho}(x;w_{1})|\notag\\
&\hspace{-0.45cm}\leq\sup_{\varrho} \E_{x}\int_{0}^{\tau_{B_{R}(0)}}\expo^{-q's}|\widetilde{h}(X_{s};w_{2})-\widetilde{h}(X_{s};w_{1})|\,\der s\notag\\
&\hspace{-0.45cm}\leq \E_{x}\int_{0}^{\infty}\expo^{-q's}\int_{B_{R+\frac{b}{2}}(0)} 2A_{0}||w_{2}-w_{1}||_{\hol^{0}(\overline{B_{R}(0)})}\nu(\der z)\,\der s\notag\\ 
&\hspace{-0.45cm}\leq \frac{2A_{0}\nu(B_{R+\frac{b}{2}}(0))}{q'}||w_{2}-w_{1}||_{\hol^{0}(\overline{B_{R}(0)})}.\hspace{2.5cm} 
\end{align*}
\end{proof}

\begin{lema}\label{contrac2}
Let $T_{\varepsilon}:\hol^{0}(\overline{B_{R}(0)})\longrightarrow\hol^{0}(\overline{B_{R}(0)})$ be as in (\ref{mapT.1.0}). Then, there exists a unique solution $w^{*}\in\hol^{0}\overline{B_{R}(0)}$ to the equation $T_{\varepsilon}(w^{*})=w^{*}$.  
\end{lema}
\begin{proof}
Recall that $T_{\varepsilon}:\hol^{0}(\overline{B_{R}(0)})\longrightarrow\hol^{0}(\overline{B_{R}(0)})$ is defined as 
\begin{equation*}
T_{\varepsilon}(w)=V^{\varepsilon}(\cdot;w),\ \text{for each}\ w\in\hol^{0}(\overline{B_{R}(0)}), 
\end{equation*}
where $V^{\varepsilon}(\cdot;w)$ is given by (\ref{value.1.0}). Then, by Hypothesis (H\ref{h4}) and Lemma \ref{contrac1}, we obtain that $T_{\varepsilon}$ is a contraction mapping in $(\hol^{0}(\overline{B_{R}(0)}),||\cdot||_{\hol^{0}(\overline{B_{R}(0)})})$. Therefore, from contraction fixed point Theorem, there exists a unique solution $w^{*}\in\hol^{0}(\overline{B_{R}(0)})$ to the equation $T_{\varepsilon}(w^{*})=w^{*}$.
\end{proof}

\subsection{Non-linear Dirichlet problem with an elliptic integro-dif\-fer\-en\-tial operator}

We begin this subsection showing the existence, regularity and uniqueness of the solution $u^{\varepsilon}$ to the non-linear integro-differential Dirichlet problem (NIDD) (\ref{p.13.0}).  To prove this, we use Lemmas \ref{convexu1}--\ref{contrac2}, stated in the previous section.

\begin{teor}\label{solNIDD.1.0}
For each $\varepsilon\in(0,1)$ fixed, there exists a unique solution $u^{\varepsilon}\in\hol^{2,\alpha}(\overline{B_{R}(0)})$ to the NIDD problem (\ref{p.13.0}).
\end{teor}
\begin{proof}
From Lemma \ref{contrac2},  there exists a unique solution $w^{*}\in\hol^{0}\overline{B_{R}(0)}$ to the equation $T_{\varepsilon}(w^{*})=w^{*}$, where $T_{\varepsilon}$ is given by (\ref{mapT.1.0}). Furthermore, we know that there exists a unique solution $\mathpzc{u}^{\varepsilon}(\cdot;w^{*})\in\hol^{2,\alpha}(\overline{B_{R}(0)})$ to the Dirichlet problem
\begin{equation*}
\begin{cases}
  q'\mathpzc{u}^{\varepsilon}(\cdot;w^{*})-\dif^{'} \mathpzc{u}^{\varepsilon}(\cdot;w^{*})+ \psi_{\varepsilon}(|\deri^{1} \mathpzc{u}^{\varepsilon}(\cdot;w^{*})|^{2})=\widetilde{h}(\cdot;w^{*}),& \text{in}\ B_{R}(0),\\
  \mathpzc{u}^{\varepsilon}(\cdot;w^{*})=0, &\text{on}\ \partial B_{R}(0),
\end{cases}
\end{equation*} 
and by Lemma \ref{convexu1}, we obtain that $\mathpzc{u}^{\varepsilon}(\cdot;w^{*})=V^{*}(\cdot;w^{*})=T_{\varepsilon}(w^{*})=w^{*}$, in $\overline{B_{R}(0)}$. Therefore, taking $u^{\varepsilon}$ as $w^{*}$, we conclude that $u^{\varepsilon}$ is in  $\hol^{2,\alpha}(\overline{B_{R}(0)})$, and it is the unique solution to the NIDD problem (\ref{p.13.0}). 
\end{proof}

\begin{rem}\label{Rem3}
Previous to this work, Bony \cite{Bony1}, Bensoussan  and Lions \cite{bens}, Lenhart \cite{lenh} and \cite{lenh2}, Gimbert and Lions \cite{gimb} and Garroni and Menaldi \cite{garroni}, among others,  studied the existence, uniqueness and regularity of the solution to the linear Dirichlet problem with an integro-differential operator similar to (\ref{p6}), obtaining results in the spaces $\sob^{2,p}$  and $\sob^{1,\infty}\cap\sob^{2,p}_{\loc}$, respectively. We note that the NIDD problem (\ref{p13.0}) is more general than the linear Dirichlet problem studied in the works mentioned above, in the sense that our problem has a  non-linear part that is determined by $\psi_{\varepsilon}(|\deri^{1}u^{\varepsilon}|^{2})$.
\end{rem}
\begin{rem}\label{rem3}
Note that the result  given in Theorem \ref{solNIDD.1.0} can be obtained in more general domains. Taking  the NIDD problem \eqref{p.13.0} on an open  and bounded set $\mathcal{O}\subset\R^{d}$, whose boundary $\partial\mathcal{O}$ is smooth, and proceeding in the same way as in Subsection \ref{NLDP.0.1} and the proof of Theorem \ref{solNIDD.1.0}, we get for each $\varepsilon\in(0,1)$ fixed, there exists a unique solution $v^{\varepsilon}\in\hol^{2,\alpha}(\overline{\mathcal{O}})$ to the NIDD problem 
\begin{equation*}
  \begin{cases}
  q'v^{\varepsilon}-\Gamma^{'}  v^{\varepsilon}+ \psi_{\varepsilon}(|\deri^{1} v^{\varepsilon}|^{2})=h, & \text{in}\ \mathcal{O},\\
   v^{\varepsilon}=0,   &\text{on}\ \partial \mathcal{O}.
  \end{cases}
  \end{equation*}
Nevertheless, we can not derive from this the solution to the HJB equation \eqref{p1}. The reason for this is that to let $\varepsilon\rightarrow0$, some uniform upper bounds of $u^{\varepsilon}$ are required. In particular, to obtain of the upper bound of $|\deri^{1}u^{\varepsilon}|$ on $B_{R}(0)$, we need to introduce an auxiliary regular function; see \eqref{desig7}, and Lemmas \ref{solintegrodif1}, \ref{lemafrontera} and \ref{lemafrontera1}. For a general domain, the determination of this function is an open problem.
\end{rem}

 \subsection{Some properties of the solution to the NIDD problem (\ref{p.13.0})}
In this subsection, we shall show some properties of the solution $u^{\varepsilon}$ to the NIDD problem (\ref{p.13.0}), such properties will in turn be used in Section 4 to establish the existence and regularity of the solution to the HJB equation (\ref{p1}). Since $h\in\hol^{2}(\overline{B_{R}(0)})$ and by a bootstrap argument we can verify that $u^{\varepsilon}\in\hol^{3,\alpha}(\overline{B_{R}(0)})$; see \cite[Thm. 3.3, Corollary 6.9 and Thm. 6.17  pp. 33, 101 and 109, respectively]{gilb}. From (\ref{p.13.0}), it is easy to verify the following proposition.   
\begin{prop}\label{intop4.1}
Let $u^{\varepsilon}$ be  the solution  to the NIDD problem (\ref{p.13.0}). Then, $u^{\varepsilon}\in\hol^{3,\alpha}(\overline{B_{R}(0)})$ and
\begin{align}
 \frac{1}{2}\tr(\sigma\deri^{2}\partial_{i}u^{\varepsilon})=q'\partial_{i}u^{\varepsilon}-\langle\deri^{1}\partial_{i}u^{\varepsilon},\widetilde{\gamma}\rangle-\partial_{i}h-\inted\partial_{i}E(u^{\varepsilon})+\psi_{\varepsilon}'(g)\partial_{i}g,\label{intop4.1.0}
\end{align}
in $B_{R}(0)$, with $i,j\in\{1,\dots,d\}$, where $g\eqdef|\deri^{1}u^{\varepsilon}|^{2}$, in $x\in B_{R}(0)$, and its first and second derivatives are, respectively,
\begin{equation}\label{intop4.1.3.0}
 \begin{cases}
\partial_{i}g=2  \sum_{k}\partial_{k}u^{\varepsilon}\partial^{2}_{ik}u^{\varepsilon},\\
 %\vspace{-0.4cm}\\
\partial^{2}_{ji}g=2  \sum_{k}\bigl(\partial^{2}_{kj}u^{\varepsilon}\partial^{2}_{ki}u^{\varepsilon}+\partial_{k}u^{\varepsilon}\partial^{3}_{jik}u^{\varepsilon}\bigr).
 \end{cases}
\end{equation}
\end{prop}
\begin{rem}
Note that from Theorem \ref{solNIDD.1.0} and Proposition \ref{intop4.1}, we obtain the result of Theorem \ref{princ1.0}.
\end{rem}

From Lemma \ref{convexu1.0}, it is easy to verify that $u^{\varepsilon}$ is a non-negative function. This fact is proved below.
\begin{prop}\label{upos1}
The solution $u^{\varepsilon}$ to the NIDD problem (\ref{p.13.0}) is a non-negative function.
\end{prop}

\begin{proof}
From the proof of Lemma \ref{convexu1.0}, it is known that 
\begin{equation*}
u^{\varepsilon}(x)= \E_{x}\biggl(\int_{0}^{t\wedge\tau_{B_{R}(0)}}\expo^{-qs}(h(\widetilde{X}_{s})+l_{\varepsilon}(\dot{\varrho}^{R}_{s}))\der s\biggr), 
\end{equation*}
where $\widetilde{X}$ and $\dot{\varrho}^{R}$ are given by (\ref{convex2}) and (\ref{convex4}), respectively. Since $h$ is a non-negative function, it follows  $u^{\varepsilon}(x)\geq \E_{x}\bigl(\int_{0}^{t\wedge\tau_{B_{R}(0)}}\expo^{-qs}h(\widetilde{Z}_{s})\der s\bigr)\geq0$. Therefore, $u^{\varepsilon}\geq0$ in $\overline{B_{R}(0)}$.
\end{proof}

Now, we shall establish estimates for 
\begin{equation*}
u^{\varepsilon},\  |\deri^{1}u^{\varepsilon}|,\  \psi_{\varepsilon}(|\deri^{1}u^{\varepsilon}|^{2})\ \text{and}\  ||\deri^{2}u^{\varepsilon}||_{\Lp^{p}(B_{r})}, 
\end{equation*}
with $B_{r}\subset B_{R}(0)$ an open ball, such that these estimates are independent of $\varepsilon$; see Lemmas \ref{inde}, \ref{lemafrontera1}, \ref{cotaphi} and \ref{lemafrontera3}. The reason for doing this  is because in Section \ref{chER} we will need to extract a convergent subsequence $\{u^{\varepsilon_{\kappa}}\}_{\kappa\geq1}$ of $\{u^{\varepsilon}\}_{\varepsilon\in(0,1)}$ such that $u\eqdef \lim_{\varepsilon_{\kappa}\rightarrow0}u^{\varepsilon_{\kappa}}$ is the  solution of the HJB equation (\ref{p1}). 

The following result is based in the weak maximum principle  for integral-differential equations. Although Theorem  \ref{princmax1.0} is valid for more general domains and integro-differential operators, see for instance \cite[Thm. 3.1.3]{garroni}, we are interested in the case that the domain and  integro-differential operator are $B_{R}(0)\subseteq\R^{d}$ and  $q-\dif -\inted^{'} E(\cdot)$, respectively, where $q>0$ and 
\begin{equation}\label{intop0}
\begin{cases}
\dif u(x)\eqdef \frac{1}{2}\tr(\sigma\deri^{2}u(x))+\langle\deri^{1}u(x),\gamma\rangle,\\
\inted^{'}E(u)(x)\eqdef \int_{\R^{*}}(E(u)(x+z)-u(x)-\deri^{1}u(x+z))\nu(\der z).
\end{cases}
\end{equation}

\begin{teor}[Weak maximum principle]\label{princmax1.0}
If 
\begin{equation*}
w\in\hol^{2}(B_{R}(0))\cap\hol^{0}_{\comp}(B_{R+\frac{b}{2}}(0)),
\end{equation*}
satisfies $qw-\dif w-\inted^{'} E(w)\leq0,\ \text{in}\ B_{R}(0)$, then 
\begin{equation*}
\sup_{\R^{d}}E(w)= \sup_{B_{R+\frac{b}{2}}(0)\setminus B_{R}(0)}[E(w)]^{+}, 
\end{equation*}
where $[E(w)]^{+}\eqdef\max\{E(w),0\}$.
\end{teor}

Note that the NIDD problem (\ref{p.13.0}) is equivalent to
\begin{equation}\label{Pd1.1.2}
  \begin{cases}
  qu^{\varepsilon}-\dif u^{\varepsilon}-\inted^{'}E(u^{\varepsilon})+ \psi_{\varepsilon}(|\deri^{1} u^{\varepsilon}|^{2})=h, & \text{in}\ B_{R}(0),\\
   u^{\varepsilon}=0,   &\text{on}\ \partial B_{R}(0).
  \end{cases}
\end{equation}
The following results will help us to establish some properties of the function $u^{\varepsilon}$, which is the solution to  the NIDD problem (\ref{p.13.0}); such properties will in turn be used in Section 4 to establish the existence of the solution to the HJB equation  (\ref{p1}).
\begin{rem}\label{solintdif1}
Observe that the linear Dirichlet problem 
\begin{equation}\label{solinte1}
  \begin{cases}
   q\eta-\dif \eta-\inted^{'} E(\eta)= h, & \text{in}\  B_{R}(0),\\
   \eta=0, & \text{on}\ \partial B_{R}(0),
  \end{cases}
\end{equation}
has a unique solution $\eta\in\hol^{2,\alpha}(\overline{B_{R}(0)})$ \cite[Thm. 3.1.12]{garroni}. We can see that the linear integro-differential Dirichlet problem (\ref{solinte1}) is equivalent to
\begin{equation*}
  \begin{cases}
   q'\eta-\dif^{'} \eta-\inted E(\eta)= h, & \text{in}\  B_{R}(0),\\
   \eta=0, & \text{on}\ \partial B_{R}(0).
  \end{cases}
\end{equation*} 
Since $h+\inted E(\eta)\in\hol^{1,\alpha}(\overline{B_{R}(0)})$ and using  similar arguments that the proof of Proposition \ref{intop4.1}, it is easy to verify that $\eta\in\hol^{3,\alpha}(\overline{B_{R}(0)})$. 
\end{rem}

\begin{lema}\label{inde}
There exists a finite constant $K_{5}>0$, independent of $\varepsilon$, such that  $u^{\varepsilon}\leq K_{5}$ in  $B_{R}(0)$.
\end{lema}

\begin{proof}
Let $u^{\varepsilon},\eta\in\hol^{3,\alpha}(\overline{B_{R}(0)})$ be solutions to (\ref{p.13.0}) and (\ref{solinte1}), respectively. Note that
\begin{equation*}
qu^{\varepsilon}-\dif u^{\varepsilon}-\inted^{'} E(u^{\varepsilon})\\
\leq qu^{\varepsilon}-\dif u^{\varepsilon}-\inted^{'} E(u^{\varepsilon})+ \psi_{\varepsilon}(|\deri^{1} u^{\varepsilon}|^{2})=h,\ \text{in}\ B_{R}(0).
\end{equation*}
Then
\begin{equation}\label{max2}
\begin{cases}
q(u^{\varepsilon}-\eta)-\dif(u^{\varepsilon}-\eta)-\inted^{'} E(u^{\varepsilon}-\eta)\leq 0,&\text{in}\ B_{R}(0),\\
(u^{\varepsilon}-\eta)=0,&\text{on}\ \partial B_{R}(0).
\end{cases}
\end{equation}
From Theorem \ref{princmax1.0}, it follows that $(u^{\varepsilon}-\eta)\leq \sup_{B_{R+\frac{b}{2}}(0)\setminus B_{R}(0)}[E(u^{\varepsilon}-\eta)]^{+}$ in $B_{R}(0)$. We prove below that $u^{\varepsilon}-\eta\leq0$ in $B_{R}(0)$. Let $x^{*}\in\overline{B_{R}(0)}$ be the point where  $u^{\varepsilon}-\eta$ in $B_{R}(0)$ attains its maximum. Observe that $(u^{\varepsilon}-\eta)(x^{*})\leq \sup_{B_{R+\frac{b}{2}}(0)\setminus B_{R}(0)}[E(u^{\varepsilon}-\eta)]^{+}$. If $x^{*}\in \partial B_{R}(0)$, we have trivially that 
\begin{equation}\label{max1}
(u^{\varepsilon}-\eta)\leq0,\ \text{in}\ B_{R}(0).
\end{equation}
Now, if $x^{*}\in B_{R}(0)$, we shall prove the statement (\ref{max1}) by contradiction. Suppose that $(u^{\varepsilon}-\eta)(x^{*})>0$. Since $u^{\varepsilon}-\eta$ attains its maximum at $x^{*}\in B_{R}(0)$ and $u-\eta=0$ on $\partial B_{R}(0)$, we have that
\begin{equation}\label{max3}
\begin{cases}
\deri^{1}(u^{\varepsilon}-\eta)(x^{*})=0,\\
\frac{1}{2}\tr(\sigma\deri^{2}(u^{\varepsilon}-\eta)(x^{*}))\leq0,\\
(u^{\varepsilon}-\eta)(x^{*}+z)-(u^{\varepsilon}-\eta)(x^{*})\leq0,\ \text{for all}\ x^{*}+z\in \overline{B_{R}(0)}.
\end{cases}
\end{equation}
Since $(u^{\varepsilon}-\eta)(x^{*}+z)-(u^{\varepsilon}-\eta)(x^{*})\leq0$, for all $x^{*}+z\in \overline{B_{R}(0)}$, and $b$ is small enough, it follows that
\begin{align}\label{max4}
0\geq\inted^{'} E(u^{\varepsilon}-\eta)(x^{*}).
\end{align}
From (\ref{max2}) and (\ref{max3}), we have that 
\begin{equation*}
0\geq\frac{1}{2}\tr(\sigma\deri^{2}(u^{\varepsilon}-\eta)(x^{*}))\geq q(u^{\varepsilon}-\eta)(x^{*})-\inted^{'} E(u^{\varepsilon}-\eta)(x^{*}). 
\end{equation*}
Then, by (\ref{max4}), we get $q(u^{\varepsilon}-\eta)(x^{*})\leq\inted^{'} E(u^{\varepsilon}-\eta)(x^{*})\leq0$, which is a contradiction and hence $u^{\varepsilon}-\eta\leq0$ in $B_{R}(0)$.  Therefore, from Remark \ref{solintdif1}, we conclude that there exists a constant $K_{5}>0$ independent of $\varepsilon$ such that $u^{\varepsilon}\leq K_{5}$ in $B_{R}(0)$, where $K_{5}\eqdef||\eta||_{\hol^{0}(\overline{B_{R}(0)})}$.
\end{proof}

Defining $\eta_{1}$ as
\begin{equation}
\eta_{1}(x)\eqdef
\begin{cases}\label{desig7}
\expo^{K_{6} R^{2}}-\expo^{K_{6}|x|^{2}}, & \text{if}\ x\in B_{R}(0),\\
E(u^{\varepsilon})(x), & \text{if}\ x\in B_{R}(0)^{\comp },
\end{cases}
\end{equation}
with $K_{6}>0$ a constant, we can see that  $\eta_{1}\in\hol^{2}(B_{R}(0))\cap\hol^{0}(B_{R}(0)^{\comp})$ is a positive concave function in $B_{R}(0)$. We have the following result.
\begin{lema}\label{solintegrodif1}
Let $\eta_{1}$ be defined as in (\ref{desig7}). Then, choosing $K_{6}>0$  large enough, 
\begin{equation}\label{desig6}
   q\eta_{1}(x)-\dif \eta_{1}(x)-\inted^{'} \eta_{1}(x)\geq C_{0}(1+|x|^{2})\geq h(x),\ \text{in}\ B_{R}(0).
\end{equation}
\end{lema}
This statement will be  helpful in finding a constant, independent of $\varepsilon$, which bounds by above $|\partial_{\vartheta}u^{\varepsilon}|$ in $\partial B_{R}(0)$. Recall that $\partial_{\vartheta}f$ denotes the directional derivative of the function $f$ with respect to the unit vector $\vartheta\in\R^{d}$, i.e.  $\partial_{\vartheta}f(x)\eqdef\lim_{\delta\rightarrow0}\frac{f(x)-f(x-\delta\vartheta)}{h},\ \text{with}\ x\in\R^{d}$.

\begin{proof}[Proof of Lemma \ref{solintegrodif1}]
Let $\eta_{1}$ be as in (\ref{desig7}). Calculating its first and second derivatives in $B_{R}(0)$,
\begin{equation}\label{primder1}
\begin{cases}
 \partial_{i} \eta_{1}(x)&=-2  K_{6}\expo^{K_{6}|x|^{2}}x_{i},\\
\partial^{2}_{ii} \eta_{1}(x)&=-2  K_{6}\expo^{K_{6}|x|^{2}}(1+2  K_{6} x_{i}^{2}),\\
\partial^{2}_{ji} \eta_{1}(x)&=-4  K_{6}^{2}\expo^{K_{6}|x|^{2}}x_{i}x_{j},
\end{cases}
\end{equation}
with $i,j\in\{1,\dots,d\}$ and $i\neq j$, by (H\ref{h3}) and (\ref{primder1}), we see that
\begin{align}\label{primder1.1}
-\dif \eta_{1}(x)&=2  K_{6}\expo^{K_{6}|x|^{2}}\biggl(\frac{1}{2}\sum_{i}\sigma_{ii}+  K_{6}\langle\sigma x,x\rangle+\langle x,\gamma\rangle\biggr)\geq2  K_{6}\expo^{K_{6}|x|^{2}}\biggl(  K_{6}\theta|x|^{2}-\Lambda|x|+\frac{\theta d}{2}\biggr).
\end{align}
Since $\eta_{1}$ is a positive concave function in $B_{R}(0)$, we have that  $\eta_{1}(x+z)-\eta_{1}(x)\leq \langle\deri^{1} \eta_{1}(x),z\rangle$, for all $|x+z|<R$. Then, using Lemma \ref{cotaintaf1}, we obtain  the following inequalities  
\begin{align}\label{primder1.2}
-\inted^{'}\eta_{1}(x)&\geq -\int_{\{|x+z|\geq R\}}(E(u^{\varepsilon})(x+z)-\eta_{1}(x)-\langle\deri^{1}\eta_{1}(x),z\rangle)\nu(\der z)\notag\\
&\geq-\biggl|\int_{\{|x+z|\geq R\}}E(u^{\varepsilon})(x+z)\nu(\der z)\biggr|+ \eta_{1}(x)\int_{\{|x+z|\geq R\}}\nu(\der z)+2  K_{6}\expo^{K_{6}|x|^{2}}\int_{\{|x+z|\geq R\}}\langle x,z\rangle\nu(\der z)\notag\\
&\geq -2K_{6}(A_{0}\nu(\mathcal{B}')+\nu_{0}\expo^{K_{6}|x|^{2}}|x|).
\end{align}
Recall that $\mathcal{B}'= B_{R+\frac{b}{2}}(0)\setminus\overline{B_{R}(0)}$ and  $\nu_{0}$, $A_{0}$ are constants given in (H\ref{h2}) and Proposition  \ref{cotEu1}, respectively.  Using (\ref{primder1.1})--(\ref{primder1.2}), we get that for any $x\in B_{R}(0)$,
\begin{equation*}
 q\eta_{1}(x)-\dif \eta_{1}(x)-\inted^{'} \eta_{1}(x)\\
 \geq 2  K_{6}\expo^{K_{6}|x|^{2}}\biggl(\theta  K_{6}|x|^{2}-(\Lambda+\nu_{0})|x|+\frac{\theta d}{2}\biggr)-2K_{6}A_{0}\nu(\mathcal{B}'),
\end{equation*}
From (H\ref{h1}) and choosing $ K_{6}$ large enough, it implies (\ref{desig6}).
\end{proof}

We obtain the following result as a consequence of the previous lemma.

\begin{lema}\label{lemafrontera}
Let $K_{6}>0$ be the constant given in Lemma \ref{solintegrodif1}. Then $|\partial_{\vartheta} u^{\varepsilon}|\leq2K_{6}R\expo^{K_{6}R^{2}}$, in $\partial B_{R}(0)$.
\end{lema}

\begin{proof}
Let $x\in\partial B_{R}(0)$, $\vartheta$ a unit vector and $\eta_{1}$ as in (\ref{desig7}). Since 
\begin{equation*}
\begin{cases}
q (u^{\varepsilon}-\eta_{1})-\dif(u^{\varepsilon}-\eta_{1})-\inted^{'}(E(u^{\varepsilon})-\eta_{1})\leq0,\ \text{in}\  B_{R}(0),\\
 \sup_{B_{R+\frac{b}{2}}\setminus B_{R}(0)}[E(u^{\varepsilon})-\eta_{1}]^{+}=0,
 \end{cases}
 \end{equation*}
by the weak maximum principle, Theorem \ref{princmax1.0}, it follows that  $u^{\varepsilon}\leq\eta_{1}$. Since these functions agree in $B_{R}(0)^{\text{c}}$ and $u^{\varepsilon}>0$, we get $\partial_{\vartheta} \eta_{1}(x)\leq\partial_{\vartheta}  u^{\varepsilon}(x)\leq0$. It implies that $|\partial_{\vartheta} u^{\varepsilon}|\leq |\deri^{1}\eta_{1}|$ in $\partial B_{R}(0)$. Recalling the definition of $\eta_{1}$ and its first derivatives, see (\ref{primder1}), it follows that $|\partial_{\vartheta} u^{\varepsilon}|\leq 2K_{6}R\expo^{K_{6}R^{2}}$ in $\partial B_{R}(0)$.
 \end{proof}

Before showing that $|\deri^{1}u^{\varepsilon}|$ is bounded by a positive constant in $\overline{B_{R}(0)}$, which is independent of $\varepsilon$; see Lemma \ref{lemafrontera1},  we  establish an auxiliary result, whose proof is in the appendix.

\begin{lema}\label{lemfrontera1.0}
Define the auxiliary function $\varphi:\overline{B_{R}(0)}\longrightarrow\R$ as 
\begin{equation}\label{aux.1.0}
\varphi\eqdef|\deri^{1}u^{\varepsilon}|^{2}-Mu^{\varepsilon},\ \text{in $\overline{B_{R}(0)}$}, 
\end{equation}
where  $M\eqdef\max_{x\in\overline{B_{R}(0)}}|\deri^{1}u^{\varepsilon}(x)|$. Then
\begin{align}\label{partu4}
\frac{1}{2}\sum_{ij}\sigma_{ij}\partial_{ij}^{2}\varphi&\geq \psi'_{\varepsilon}(g)(2\langle\deri^{1}\varphi,\deri^{1}u^{\varepsilon}\rangle+M|\deri^{1}u^{\varepsilon}|^{2})-(K_{8}+MK_{9})|\deri^{1}u^{\varepsilon}|-MK_{7}-\langle\deri^{1}\varphi,\widetilde{\gamma}\rangle,
\end{align}
in $B_{R}(0)$, where the constants $K_{7},K_{8},K_{9}$ are independent of $\varepsilon$.
\end{lema}

\begin{lema}\label{lemafrontera1}
There exists a constant $K_{10}>0$ independent of $\varepsilon$  such that $|\deri^{1}u^{\varepsilon}|\leq K_{10}$, in $\overline{B_{R}(0)}$. 
\end{lema}

\begin{proof}
Consider the auxiliary function $\varphi$ as in (\ref{aux.1.0}). Observe that if $M\leq1$, we obtain a bound for $M$ that is independent of $\varepsilon$.  We assume henceforth that $M\geq1$. Taking $x^{*}\in\overline{B_{R}(0)}$ as a point where $\varphi$ attains its maximum on $B_{R}(0)$, it  suffices to bound $|\deri^{1}u^{\varepsilon}(x^{*})|^{2}$ for a constant independent of $\varepsilon$, since
\begin{align}\label{partu5}
|\deri^{1}u^{\varepsilon}(x)|^{2}&\leq |\deri^{1}u^{\varepsilon}(x^{*})|^{2}+M(u^{\varepsilon}(x^{*})+u^{\varepsilon}(x))\notag\\
&\leq|\deri^{1}u^{\varepsilon}(x^{*})|^{2}+2MK_{5},
\end{align}
for all $x\in \overline{B_{R}(0)}$. The last inequality in (\ref{partu5}) is obtained from Lemma \ref{inde}.  If $x^{*}\in\partial B_{R}(0)$,  by  Lemma \ref{lemafrontera}, it is easy to  deduce  $\varphi(x^{*})=|\deri^{1}u^{\varepsilon}(x^{*})|^{2}\leq 2K_{6}R\expo^{K_{6}R^{2}}$, where $K_{6}$ is as in Lemma \ref{lemafrontera}. Then, from (\ref{partu5}), $|\deri^{1}u^{\varepsilon}|^{2}\leq2K_{6}R\expo^{K_{6}R^{2}}+2MK_{5}$, in $\overline{B_{R}(0)}$. Note that for all $\epsilon$, there exists $x_{0}\in\overline{B_{R}(0)}$ such that $(M-\epsilon)^{2}\leq|\deri^{1} u^{\varepsilon}(x_{0})|^{2}$. Then 
\begin{align}\label{partu5.0}
(M-\epsilon)^{2}\leq2K_{6}R\expo^{K_{6}R^{2}}+2MK_{5}.
\end{align}
Letting $\epsilon\rightarrow0$ in (\ref{partu5.0}), it follows
$M\leq2K_{6}R\expo^{K_{6}R^{2}}+2K_{5}$. When $x^{*}\in B_{R}(0)$,  we have that $\deri^{1}\varphi(x^{*})=0$ and $\frac{1}{2}\sum_{ij}\sigma_{ij}\partial_{ij}\varphi(x^{*})\leq0$. Then, from (\ref{partu4}), we get 
\begin{equation}\label{partu6}
0\geq M\psi'_{\varepsilon}(g(x^{*}))|\deri^{1}u^{\varepsilon}(x^{*})|^{2}-(K_{8}+MK_{9})|\deri^{1}u^{\varepsilon}(x^{*})|-MK_{7}.
\end{equation}
If $\psi'_{\varepsilon}(g(x^{*}))<1<\frac{1}{\varepsilon}$, by definition of $\psi_{\varepsilon}$, given in (\ref{p12.2}), we obtain that $\psi_{\varepsilon}(g(x^{*}))\leq1$. It follows that  $|\deri^{1}u^{\varepsilon}(x^{*})|\leq2\varepsilon+1\leq2$. Then, by (\ref{partu5}) and arguing as  in (\ref{partu5.0}), we obtain $M\leq4+2K_{5}$. If $\psi'_{\varepsilon}(g(x^{*}))\geq1$, from (\ref{partu6}), we get  
\begin{equation*}
0\geq M|\deri^{1}u^{\varepsilon}(x^{*})|^{2}-(K_{8}+MK_{9})|\deri^{1}u^{\varepsilon}(x^{*})|-MK_{7},
\end{equation*}
and hence it yields
\begin{align*}
|\deri^{1}u^{\varepsilon}(x^{*})|\leq\frac{K_{8}+MK_{9}+((K_{8}+MK_{9})^{2}+4M^{2}K_{7})^{\frac{1}{2}}}{2M}\leq K_{8}+K_{9}+K_{7}^{\frac{1}{2}}.
\end{align*}
Using (\ref{partu5}) and a similar argument that (\ref{partu5.0}), we conclude $M\leq \bigl(K_{8}+K_{9}+K_{7}^{\frac{1}{2}}\bigr)^{2}+2K_{5}$. Therefore, there exists a constant $K_{10}>0$, independent of $\varepsilon$,  such that $|\deri^{1}u^{\varepsilon}|\leq K_{10}$ in $\overline{B_{R}(0)}$ since the constants $K_{5},\dots,K_{9}$ are independent of $\varepsilon$.
\end{proof}

In Lemma \ref{cotaphi}, we shall establish that $\psi_{\varepsilon}(|\deri^{1}u^{\varepsilon}|^{2})$ is locally bounded by a constant independent of $\varepsilon$.  Previous, we give an auxiliary result, whose proof is in the appendix.  
\begin{lema}\label{cotaphi.1.0}
For each cutoff function $\xi$  in $\hol^{\infty}_{\text{c}}(B_{r})$ satisfying $0\leq \xi\leq 1$, with $B_{r}\subset B_{R}(0)$, define the function $\phi:\overline{B_{r}}\longrightarrow\R$ as
\begin{equation}\label{phi.1.0}
    \phi(x)=\xi(x)\psi_{\varepsilon}(g(x)).
  \end{equation} 
Then,
\begin{align}\label{psic1}
\frac{1}{2}\tr(\sigma\deri^{2}\phi)&\geq -K_{11}(K_{12}|\deri^{2}u^{\varepsilon}|+K_{13})+\psi'_{\varepsilon}(g)(\theta\xi|\deri^{2}u^{\varepsilon}|^{2}\notag\\
&\quad-K_{14}|\deri^{2}u^{\varepsilon}|-K_{15}+2\langle\deri^{1}\phi,\deri^{1}u^{\varepsilon}\rangle),
\end{align}
in $B_{r}$, where $K_{11},\dots,K_{15}$ are positive constants independent of $\varepsilon$.
\end{lema}

\begin{lema}\label{cotaphi}
Let $B_{r}\subset B_{R}(0)$ be an open ball. For each $\xi\in\hol_{\comp }^{\infty}(B_{r})$ satisfying $0\leq\xi\leq1$,  there exist positive constants $K_{12},\dots,K_{15}$ independent of $\varepsilon$, such that
  \begin{align}\label{cotaphi1.0}
    \xi\psi_{\varepsilon}(|\deri^{1} u^{\varepsilon}|^{2})\leq\frac{K_{12}(K_{14}+(\theta K_{15})^{\frac{1}{2}})}{\theta}+K_{13},\ \text{in}\ B_{r}.  
  \end{align}
The constant $\theta>0$ is as in Hypothesis (H\ref{h3}). 
\end{lema}

\begin{proof}
Let $B_{r}\subset B_{R}(0)$ and for each cutoff function $\xi$ in $\hol^{\infty}_{\text{c}}(B_{r})$ satisfying $0\leq \xi\leq 1$, define $\phi$ as in (\ref{phi.1.0}). Taking $x^{*}\in\overline{B_{r}}$ as a point where $\phi$ attains its maximum on $B_{r}$, it  suffices to bound $\phi(x^{*})$ by a constant independent of $\varepsilon$. If $x^{*}\in\partial B_{r}$ then $\phi(x)\leq\phi(x^{*})=0$. When $x^{*}\in B_{R}(0)$,  we have $\deri^{1} \phi(x^{*})=0$ and $\frac{1}{2}\tr(\sigma\deri^{2}\phi(x^{*}))\leq0$. Then, from (\ref{psic1}), we get that
\begin{align}\label{psic1.0}
0\geq& -K_{11}(K_{12}|\deri^{2}u^{\varepsilon}(x^{*})|+K_{13})\notag\\
&\quad+\psi'_{\varepsilon}(g(x^{*}))(\theta\xi(x^{*})|\deri^{2}u^{\varepsilon}(x^{*})|^{2}-K_{14}|\deri^{2}u^{\varepsilon}(x^{*})|-K_{15}),
\end{align}
where $K_{11},\dots,K_{15}$ are constants as in Lemma \ref{cotaphi.1.0}. If $\psi'_{\varepsilon}(g(x^{*}))\leq1<\frac{1}{\varepsilon}$, by the definition of $\psi_{\varepsilon}$,  given in   (\ref{p12.1}), we obtain that  $\psi_{\varepsilon}(g(x^{*}))\leq1$. Then, $\phi(x)\leq\phi(x^{*})\leq1$. In the case where $\psi'_{\varepsilon}(g(x^{*}))\geq1$, from (\ref{psic1.0}), we get that 
\begin{equation*}
0\geq \theta\xi(x^{*})|\deri^{2}u^{\varepsilon}(x^{*})|^{2}-K_{14}|\deri^{2}u^{\varepsilon}(x^{*})|-K_{15}, 
\end{equation*}
and hence it follows $|\deri^{2}u^{\varepsilon}(x^{*})|\leq\frac{K_{14}+(K_{14}^{2}+4\theta\xi(x^{*})K_{15})^{\frac{1}{2}}}{2\theta\xi(x^{*})}$. Therefore, from this and (\ref{cotgrad1}), we conclude \eqref{cotaphi1.0}. We finish the proof. 
\end{proof}

\begin{lema}\label{lemafrontera3}
Let $1\leq p< \infty$ and $\beta\in(0,1)$ such that $B_{\beta'r}\subset B_{R}(0)$, with  $\beta'= \frac{\beta+1}{2}$. There exists a constant $K_{18}=K_{18}(\beta r,p)>0$ independent of $\varepsilon$  such that
\begin{align}\label{eqW1.1}
    ||\deri^{2}u^{\varepsilon}||_{\Lp^{p}(B_{\beta r})}&\leq K_{18}\bigr(||h||_{\Lp^{p}(B_{\beta' r})}+||\inted E(u^{\varepsilon})||_{\Lp^{p}(B_{\beta' r})}+||\deri^{1}u^{\varepsilon}||_{\Lp^{p}(B_{\beta' r})}\notag\\
  &\quad+||\xi\psi_{\varepsilon}(|\deri^{1}u^{\varepsilon}|)||_{\Lp^{p}(B_{\beta' r})}+||u^{\varepsilon}||_{\Lp^{p}(B_{\beta' r})}\bigl),
\end{align}
with  $\beta'=2^{-1}(\beta+1)$.
\end{lema}

\begin{proof}
Let $r>0$, $\beta\in(0,1)$ and  $\xi\in\hol^{\infty}_{\comp}(B_{r})$ a cutoff function such that $0\leq\xi\leq1$, $\xi=1$ on $B_{\beta r}$ and $\xi=0$ on $B_{\beta' r}^{\comp}$, with $\beta'=\frac{\beta+1}{2}$. Suppose that  $|\deri^{1}\xi|\leq K_{16}$ and $|\deri^{2}\xi|\leq K_{16}$, for some constant $K_{16}>0$. Defining $w=\xi u^{\varepsilon}$, we obtain 
\begin{align}\label{eqW1}
||\deri^{2}u^{\varepsilon}||_{\Lp^{p}(B_{\beta r})}&\leq ||\deri^{2}u^{\varepsilon}||_{\Lp^{p}(B_{\beta r})}+||\deri^{2}\xi u^{\varepsilon}||_{\Lp^{p}(B_{\beta' r}\setminus B_{\beta r})}\notag\\
&= ||\deri^{2}w||_{\Lp^{p}(B_{\beta' r})}.
\end{align}
Calculating first and second derivatives of $w$ in $B_{\beta'r}$,
\begin{align*}
\partial_{i}w&=u^{\varepsilon}\partial_{i}\xi+\partial_{i}u^{\varepsilon}\xi,\\
\partial_{ji}^{2}w&=\partial_{j}u^{\varepsilon}\partial_{i}\xi+u^{\varepsilon}\partial_{ji}^{2}\xi+\partial_{i}u^{\varepsilon}\partial_{j}\xi+\xi\partial_{ji}^{2}u^{\varepsilon},
\end{align*}
with $j,i\in\{1,\dots,d\}$,  by  (\ref{p.13.0}), we get that 
\begin{equation}\label{dirif1}
\begin{cases}
q'w- \dif' w=f, &\text{in}\ B_{\beta'r}, \\
w=0, & \text{on}\ \partial B_{\beta'r},
\end{cases}
\end{equation}
where
\begin{align}\label{eqW1.2}
f&\eqdef\xi(h+\inte E(u^{\varepsilon})-\psi_{\varepsilon}(|\deri^{1} u^{\varepsilon}|^{2}))\notag\\
&\quad-u^{\varepsilon}\biggr(\frac{1}{2}\tr(\sigma\deri^{2}\xi)+\langle\deri^{1} \xi,\widetilde{\gamma}\rangle\biggl)-\langle\sigma\deri^{1}\xi,\deri^{1} u^{\varepsilon}\rangle.
\end{align}
We know that for  the linear Dirichlet problem (\ref{dirif1})  (see \cite[Lemma 3.1]{lenh}), there exists a constant $K_{17}=K_{17}(\beta r,p)>0$ independent of $w$, such that $||\deri^{2}w||_{\Lp^{p}(B_{\beta' r})}\leq K_{17}||f||_{\Lp^{p}(B_{\beta' r})}$. Estimating the terms on the right hand side of  (\ref{eqW1.2}) with the norm $||\cdot||_{\Lp^{p}(B_{\beta' r})}$ and by the choice of $\xi$, it follows
\begin{align}\label{eqW1.3}
||\deri^{2}w||_{\Lp^{p}(B_{\beta' r})}&\leq K_{18}\bigr(||h||_{\Lp^{p}(B_{\beta' r})}+||\inted E(u^{\varepsilon})||_{\Lp^{p}(B_{\beta' r})}+||\deri^{1}u^{\varepsilon}||_{\Lp^{p}(B_{\beta' r})}\notag\\
  &\quad+||\xi\psi_{\varepsilon}(|\deri^{1}u^{\varepsilon}|^{2})||_{\Lp^{p}(B_{\beta' r})}+||u^{\varepsilon}||_{\Lp^{p}(B_{\beta' r})}\bigl),
\end{align}
for some constant $K_{18}=K_{18}(\beta r,p)>0$ independent of $\varepsilon$. Hence, from (\ref{eqW1}) and (\ref{eqW1.3}), we have the inequality (\ref{eqW1.1}).
\end{proof}

By (\ref{normw}) and Lemma \ref{lemafrontera3} it is easy to obtain  the following result:
\begin{lema}\label{lemafrontera4}
Let  $1\leq p<\infty$ and $\beta\in(0,1)$ such that $B_{\beta'r}\subset B_{R}(0)$, with  $\beta'= \frac{\beta+1}{2}$. There exists a  constant $K_{19}>0$ independent of $\varepsilon$ such that
\begin{align*}
  ||u^{\varepsilon}||_{\text{W}^{2,p}(B_{\beta r})}&\leq K_{19}(||h||_{\Lp^{p}(B_{\beta' r})}+||\inted E(u^{\varepsilon})||_{\Lp^{p}(B_{\beta' r})}\notag\\
  &\quad+||\xi\psi_{\varepsilon}(|\deri^{1}u^{\varepsilon}|)||_{\Lp^{p}(B_{\beta' r})}+||\deri^{1}u^{\varepsilon}||_{\Lp^{p}(B_{\beta' r})}+||u^{\varepsilon}||_{\Lp^{p}(B_{\beta' r})}),
\end{align*}
\textit{with  $\beta'=\frac{\beta+1}{2}$.}
\end{lema}

\section[Existence, uniqueness and regularity]{Existence, uniqueness and regularity to the HJB equation (\ref{p.1.0})}\label{chER}

In this section, we shall present  the proof of Theorem \ref{princ1.0.1}. Note that the HJB equation (\ref{p.1.0}) can be written as
\begin{equation}\label{HJB2}
\begin{cases}
  \max\{qu-\dif u-\inted^{'}E(u) -h,|\deri^{1} u|^{2}-1\}=0,&\text{in}\ B_{R}(0),\\
  u=0, &\text{on}\ \partial B_{R}(0),
\end{cases}  
\end{equation}
where $q>0$, $\dif$ and $\inte^{'}$ are defined as in \eqref{intop0}. In order to prove Theorem \ref{princ1.0.1}, first we shall verify the existence and regularity of the solution to HJB equation (\ref{p.1.0}). Finally, we shall prove the uniqueness of the solution to the HJB equation (\ref{p.1.0}). To verify this last part, we use  Bony's maximum principle \cite{lions}.

Before the proof of Theorem \ref{princ1.0.1}, we shall introduce some preliminary results. By Lemmas \ref{inde} and \ref{lemafrontera1}, we obtain that there exists a constant $K_{20}>0$ independent of $\varepsilon$, such that
\begin{align}
||u^{\varepsilon}||_{\hol^{0,1}(\overline{B_{R}(0)})}&<K_{20},\ \text{for all}\ \varepsilon\in(0,1),\label{sob1}
\end{align}
Moreover, Proposition \ref{cotEu1} and Lemmas \ref{inde}, \ref{lemafrontera1}, \ref{cotaphi}--\ref{lemafrontera4}, guarantee that for each $B_{r}\subset B_{R}(0)$ there exist positive constants $K_{21},K_{22}$ independent of $\varepsilon$ such that
\begin{equation}\label{sob2}
\begin{cases}
||\deri^{2}u^{\varepsilon}||_{\Lp^{p}(B_{\beta r})}\leq K_{21},\\
||u^{\varepsilon}||_{\sob^{2,p}(B_{\beta r})}<K_{22},
\end{cases}
\end{equation}
for all $\varepsilon\in(0,1)$, where $\beta\in(0,1)$ and $1\leq p<\infty$ fixed. Finally, if we take $d<p<\infty$ in (\ref{sob2}), then, from Sobolev embedding  Theorem \cite[Thm. 4.12, p. 85]{adams}, we have that for each $B_{r}\subset B_{R}(0)$, there exists a positive constant $K_{23}$ independent of $\varepsilon$ such that
\begin{equation}\label{hol.1.1.0.1}
||u^{\varepsilon}||_{\hol^{1,\alpha'}(\overline{B_{\beta r}})}\leq K_{23},\ \text{for all}\ \varepsilon\in(0,1),
\end{equation}
with $\beta\in(0,1)$ fixed and $\alpha'=1-\frac{d}{p}$. 

As a consequence of Arzel\`a-Ascoli Theorem and the reflexivity of $\Lp^{p}(B_{r})$; see \cite[Thm. 7.25, p. 158 and Thm. 2.46, p. 49, respectively]{rudin, adams},  and (\ref{sob1})--(\ref{hol.1.1.0.1}), we obtain the following key result.
\begin{lema}\label{converu1}
Let $d< p<\infty$. There exist a decreasing subsequence $\{\varepsilon_{\kappa(\iota)}\}_{\iota\geq1}$, with $\varepsilon_{\kappa(\iota)}\underset{\iota\rightarrow\infty}{\longrightarrow}0$, and  $u\in\hol^{0,1}(\overline{B_{R}(0)})\cap\sob^{2,p}_{\loc}(B_{R}(0))$ such that 
\begin{equation}\label{conver.1.0}
   \begin{cases}
   u^{\varepsilon_{\kappa(\iota)}}\underset{\varepsilon_{\kappa(\iota)}\rightarrow0}{\longrightarrow}u, &\text{in}\ \hol^{0}(\overline{B_{R}(0)}),\\
   \partial_{i}u^{\varepsilon_{\kappa(\iota)}}\underset{\varepsilon_{\kappa(\iota)}\rightarrow0}{\longrightarrow}\partial_{i}u, &\text{in}\ \hol^{0}_{\loc}(B_{R}(0)),\ i\in\{1,\dots,d\},\\
   \partial_{ij}^{2}u^{\varepsilon_{\kappa(\iota)}}\underset{\varepsilon_{\kappa(\iota)}\rightarrow0}{\longrightarrow}\partial^{2}_{ij}u, &\text{weakly  in}\ \Lp^{p}_{\loc}(B_{R}(0)),\ i,j\in\{1,\dots,d\},
   \end{cases}
  \end{equation}
  where $\partial_{ij}u$ represents the second weakly derivative of $u$, with $i,j\in\{1,\dots,d\}$. Moreover, the following convergence also holds
\begin{equation}\label{conver1.0.3}
\inted E(u^{\varepsilon_{\kappa(\iota)}})\underset{\varepsilon_{\kappa(\iota)}\rightarrow0}{\longrightarrow}\inted E (u_{r}),\ \text{uniformly in $B_{R}(0)$}. 
\end{equation}
\end{lema}

\subsection[Proof. Existence and uniqueness]{Proof of Theorem \ref{princ1.0.1}}\label{extuniqu1}

We proceed to show the existence and uniqueness to the solution of the HJB equation (\ref{p.1.0}).
\begin{proof}[Proof of Theorem \ref{princ1.0.1}. Existence and regularity]
Let $d<p<\infty$. From Lemma \ref{converu1}, we know  that there exist a decreasing subsequence $\{\varepsilon_{\kappa(\iota)}\}_{\iota\geq1}$, with $\varepsilon_{\kappa(\iota)}\underset{\iota\rightarrow\infty}{\longrightarrow}0$, and  $u\in\hol^{0,1}(\overline{B_{R}(0)})\cap\sob^{2,p}_{\loc}(B_{R}(0))$ satisfying (\ref{conver.1.0}) and (\ref{conver1.0.3}). Let $\phi$ be a non-negative function in $\hol^{\infty}_{\comp}(B_{r})$, where $B_{r}\subset B_{R}(0)$.  Since for each $\varepsilon_{\kappa(\iota)}\in(0,1)$, the function $u^{\varepsilon_{\kappa(\iota)}}$ is the unique solution of  the NIDD problem (\ref{p.13.0}), we get
\begin{equation}\label{ineq.1.1}
\int_{B_{r}}(q'u^{\varepsilon_{\kappa(\iota)}}-\dif^{'} u^{\varepsilon_{\kappa(\iota)}}-\inted E(u^{\varepsilon_{\kappa(\iota)}}))\phi\der x\leq \int_{B_{ r}}h\phi\der x.
\end{equation}
Furthermore, from Lemma \ref{converu1}, we have
\begin{align}\label{ineq.1.1.0}
\biggl|\int_{B_{r}}&(q'(u^{\varepsilon_{\kappa(\iota)}}-u)-\dif^{'}(u^{\varepsilon_{\kappa(\iota)}}-u)-\inte E(u^{\varepsilon_{\kappa(\iota)}}-u))\phi\der x\biggr|\underset{\varepsilon_{\kappa(\iota)}\rightarrow0}{\longrightarrow}0.
\end{align}
Then, from \eqref{ineq.1.1} and \eqref{ineq.1.1.0}, we obtain 
\begin{equation}\label{ineq.1.1.1}
\int_{B_{ r}}(q'u-\dif^{'} u-\inted E(u))\phi\der x\leq \int_{B_{ r}}h\phi\der x.
\end{equation}
Since \eqref{ineq.1.1.1} holds for any non-negative function $\phi\in\hol^{\infty}_{\comp}(B_{ r})$ and any open ball $B_{r}\subset B_{R}(0)$, it follows that
\begin{equation}\label{solhgb1}
q'u-\dif^{'} u-\inted E(u)\leq h,\ \text{a.e. in}\ B_{R}(0).
\end{equation}
Now, since $\psi_{\varepsilon}(|\deri^{1} u^{\varepsilon_{\kappa(\iota)}}|^{2})$ is locally uniform bounded; see Lemma \ref{cotaphi}, independently of   $\varepsilon_{\kappa(\iota)}$, we have 
\begin{equation}\label{solhgb.1}
|\deri^{1} u|^{2}\leq1,\  \text{in}\ B_{R}(0). 
\end{equation}
Suppose that $|\deri^{1} u(x^{*})|^{2}<1$, for some $x^{*}\in B_{R}(0)$. Then, by the continuity of  $\deri^{1} u$, there exists a small open ball $B_{r}\subset B_{R}(0)$ such that $x^{*}\in B_{r}$ and    
$|\deri^{1} u|^{2}<1$ in $ B_{r}$. Since $\deri^{1} u^{\varepsilon_{\kappa(\iota)}}\underset{\varepsilon_{\kappa(\iota)}\rightarrow0}{\longrightarrow}\deri^{1} u$ uniformly in $B_{r}$, we obtain that there exists $\varepsilon^{\kappa(\iota_{0})}\in(0,1)$ such that for $\varepsilon_{\kappa(\iota)}\leq \varepsilon_{\kappa(\iota_{0})}$,  $|\deri^{1} u^{\varepsilon_{\kappa(\iota)}}|^{2}<1$ in $B_{r}$. Then, from (\ref{p.13.0}) and the definition of $\psi_{\varepsilon}$, it follows that for $\varepsilon_{\kappa(\iota)}\leq \varepsilon_{\kappa(\iota_{0})}$,  $q'u^{\varepsilon_{\kappa(\iota)}}-\dif^{'} u^{\varepsilon_{\kappa(\iota)}}-\inted E(u^{\varepsilon_{\kappa(\iota)}})= h$ in $B_{r}$. Then,
\begin{equation}\label{ineq.1.2}
\int_{B_{  r}}(q'u^{\varepsilon_{\kappa(\iota)}}-\dif^{'} u^{\varepsilon_{\kappa(\iota)}}-\inted E(u^{\varepsilon_{\kappa(\iota)}}))\phi\der x= \int_{B_{ r}}h\phi\der x,
\end{equation}
for any non-negative function $\phi$ in $\hol^{\infty}_{\comp}(B_{ r})$. From \eqref{ineq.1.1.0} and  \eqref{ineq.1.2}, we obtain 
\begin{equation*}
\int_{B_{ r}}(q'u-\dif^{'} u-\inted E(u))\phi\der x= \int_{B_{ r}}h\phi\der x,
\end{equation*}
for any non-negative function $\phi$ in $\hol^{\infty}_{\comp}(B_{  r})$. Therefore,
\begin{equation}\label{solhgb.2}
q'u-\dif' u-\inted E(u)= h,\ \text{a.e. in}\  B_{  r}.  
\end{equation}
Finally, since $u^{\varepsilon_{\kappa(\iota)}}=0$ on $\partial B_{R}(0)$ and $u^{\varepsilon_{\kappa(\iota)}}\underset{\varepsilon_{\kappa(\iota)}\rightarrow0}{\longrightarrow}u$ uniformly in $\overline{B_{R}(0)}$, we have  
\begin{equation}\label{solhgb2}
u=0,\ \text{on}\  \partial B_{R}(0).
\end{equation}
From (\ref{solhgb1}), (\ref{solhgb.1}), (\ref{solhgb.2}) and (\ref{solhgb2}), we conclude that $u$ is a solution to the HJB equation (\ref{p.1.0}) a.e. in $B_{R}(0)$.
\end{proof}

\begin{proof}[Proof of Theorem \ref{princ1.0.1}. Uniqueness]
To show the uniqueness of the HJB equation (\ref{p.1.0}), we shall  use the HJB equation (\ref{HJB2}) which is equivalent to it. Let  $d<p <\infty$. Suppose that there exist $u_{1},u_{2}\in\hol^{0,1}(\overline{B_{R}(0)})\cap\sob^{2,p}_{\loc}(B_{R}(0))$ two solutions to the HJB equation (\ref{HJB2}). Let $x^{*}\in\overline{B_{R}(0)}$ be the point where $u_{1}-u_{2}$ attains its maximum. If $x^{*}\in\partial B_{R}(0)$, it is easy to see 
\begin{equation}\label{HJB3}
(u_{1}-u_{2})(x)\leq(u_{1}-u_{2})(x^{*})=0,\ \text{in}\ B_{R}(0). 
\end{equation}
If $x^{*}\in B_{R}(0)$, we shall prove (\ref{HJB3}) by contradiction. Suppose $(u_{1}-u_{2})(x^{*})>0$. For $\rho>0$ small enough, the function $f\eqdef(1-\rho)u_{1}-u_{2}$, defined on $\overline{B_{R}(0)}$, is positive at some point of $B_{R}(0)$, with $f=0$ on $\partial B_{R}(0)$, and hence that $f(x^{*}_{1})>0$, where  $x^{*}_{1}\in B_{R}(0)$ is the point  where  $f$ attains its maximum. Besides, we have
\begin{equation*}
\begin{cases}
\deri^{1}f(x^{*}_{1})=0,\\
f(x^{*}_{1}+z)\leq f(x^{*}_{1}),\ \text{for all}\ x^{*}_{1}+z\in B_{R}(0). 
\end{cases}
\end{equation*}
Since $f(x^{*}_{1}+z)\leq f(x^{*}_{1})$ for all $x^{*}_{1}+z\in B_{R}(0)$, it follows that
\begin{align*}
0&\geq\inted^{'} E(f)(x^{*}_{1})\\
&=\int_{\R^{*}}(f(x^{*}_{1}+z)-f(x^{*}_{1}))\uno_{B_{R}(0)}(x^{*}_{1}+z)\nu(\der z)+\int_{\R^{*}}(E(f)(x^{*}_{1}+z)-f(x^{*}_{1}))\uno_{\mathcal{B}'}(x^{*}_{1}+z)\nu(\der z),
\end{align*}
with $\mathcal{B}'= B_{R+\frac{b}{2}}(0)\setminus\overline{B_{R}(0)}$. Since $\deri^{1}f(x^{*}_{1})=0$, $|\deri^{1}u_{1}(x^{*}_{1})|\leq1$ and $\rho>0$, we get that $|\deri^{1}u_{2}(x^{*}_{1})|=(1-\rho)|\deri^{1}u_{1}(x^{*}_{1})|<1$. This implies that  there exists $\mathcal{V}_{x^{*}_{1}}$ a neighbourhood of $x_{1}^{*}$ such that     
\begin{equation*}
\begin{cases}
qu_{2}(x)-\dif u_{2}(x)-\inted^{'}E(u_{2})(x)=h(x),\\
qu_{1}(x)-\dif u_{1}(x)-\inted^{'}E(u_{1})(x) \leq h(x),
\end{cases}
\text{for all}\ x\in\mathcal{V}_{x_{1}^{*}}.
\end{equation*}
Then, $qf-\dif f -\inted^{'}E(f)\leq-\rho h$  in $\mathcal{V}_{x_{1}^{*}}$, and hence,
\begin{equation*}
\dfrac{1}{2}\tr(\sigma\deri^{2}f)\geq q f-\inted^{'} E(f)-\langle\deri^{1}f,\gamma\rangle+\rho h,\ \text{in}\ \mathcal{V}_{x_{1}^{*}}.
\end{equation*}
Using Bony's maximum principle; see \cite{lions}, it yields
\begin{align*}
0&\geq\limess_{x\rightarrow x_{1}^{*}} \dfrac{1}{2}\tr(\sigma\deri^{2}f)(x)\geq qf(x_{1}^{*})-\inted^{'} E(f)(x_{1}^{*})+\rho h(x_{1}^{*}),
\end{align*}
which is a contradiction. The application of Bony's maximum principle is permitted here because $u_{1},u_{2}\in \sob^{2,p}_{\loc}(B_{R}(0))$ and $d< p<\infty$. Therefore, we have $u_{1}-u_{2}\leq0$ in $B_{R}(0)$. Taking $u_{2}-u_{1}$ and proceeding in a similar way as before, it follows that $u_{2}-u_{1}\leq 0$,  in $x\in B_{R}(0)$, and hence  we conclude that the solution $u$ to the HJB equation (\ref{p.1.0}) is unique.
\end{proof}

\renewcommand\thesection{A}
\setcounter{equation}{0}
\section*{Appendix. Proofs of some technical results}\label{techn1}

\begin{proof}[Proof of Lemma \ref{verif}]
Let $X=\{X_{t}:t\geq0\}$ be the state process as in (\ref{p3}), and we  assume that $u$   is a convex function in $\hol^{2}(\R^{d})$, such that it is a solution of the HJB equation (\ref{p1.0.4}). To prove (i) we consider an initial state $x\in\R^{d}$ and  a control process $(N,\xi)$. Using integration by parts in $\expo^{-qt}u(X_{t})$ \cite[Cor. 2, p. 68]{pro} and applying It\^o's formula to $u(X_{t})$  \cite[Thm. 33, p. 81]{pro}, it follows that
\begin{align}\label{expand}
  \expo^{-qt}u(X_{t})-u(x)&=\int_{0}^{t}\expo^{-qs}((\Gamma-q)u(X_{s})+h(X_{s})) \der s\notag\\
  &\hspace{-1cm}-\int_{0}^{t}\expo^{-qs}h(X_{s}) \der s+\int_{0}^{t}\expo^{-qs}\langle\deri^{1} u(X_{s-}),N_{s}\rangle\der\xi_{s}^{\comp}+M_{t}\notag\\
  &\hspace{-1cm}+\sum_{0<s\leq t}\expo^{-qs}(u(A_{s}+N_{s}\Delta \xi_{s})-u(A_{s}))\uno_{\{|\Delta\xi_{s}|\neq0,\,|\Delta Y_{s}|=0\}},
\end{align}
for all $t\geq0$, where $\xi^{\comp}$ is the contimuous part of $\xi$, and
\begin{align*}	
A_{t}&\eqdef X_{t-}+\Delta Y_{t},\\
  M_{t}&\eqdef\int_{0}^{t}\expo^{-qs}\langle\deri^{1} u(X_{s}),\der W_{s}\rangle\\
  &\quad+\int_{0}^{t}\int_{\R^{*}}\expo^{-qs}(u(X_{s-}+z)-u(X_{s-}))(\vartheta(\der s\times\der z)-\nu(\der z)\der s).
\end{align*}
Since the process $M=\{M_{t}:t\geq0\}$ is a local martingale and defining the stopping time $\tau_{B_{n}(0)}$ as $\tau_{B_{n}(0)}=\inf\{t>0 :X_{t}\notin B_{n}(0)\}$, for all $n\geq1$, the process $M^{\tau_{B_{n}}(0)}=\{M_{t\wedge \tau_{B_{n}(0)}}:t\geq0\}$ is a $\Pro_{x}$-martingale with $M_{0}=0$. Then, taking expected value in (\ref{expand}), it follows that
\begin{align}\label{dem9}
  u(x)&=\E_{x}\bigl(\expo^{-q(t\wedge\tau_{B_{n}(0)})}u(X_{t\wedge\tau_{B_{n}(0)} })\bigr)+\E_{x}\Biggl(\int_{0}^{t\wedge\tau_{B_{n}(0)}}\expo^{-qs}h(X_{s}) \der s\Biggr)\notag\\
  &+\E_{x}\Biggl(\int_{0}^{t\wedge\tau_{B_{n}(0)} }\expo^{-qs}((q-\Gamma)u(X_{s})-h(X_{s})) \der s\Biggr)\notag\\
  &-\E_{x}\biggr(\int_{0}^{t\wedge\tau_{B_{n}(0)}}\expo^{-qs}\langle\deri^{1} u(X_{s-}),N_{s}\rangle\der\xi_{s}^{\comp}\biggl)\notag\\
  &-\E_{x}\Biggl(\sum_{0<s\leq t\wedge\tau_{B_{n}(0)}}\expo^{-qs}(u(A_{s}+N_{s}\Delta \xi_{s})-u(A_{s}))\uno_{\{|\Delta\xi_{s}|\neq0,\,|\Delta Y_{s}|=0\}}\Biggr).
\end{align}
Given that $u$ is a convex solution to the HJB equation (\ref{p1.0.4}), we know that
\begin{equation*}
\begin{cases}
|\deri^{1} u(X_{t-})|^{2}-1\leq0,\\
(q-\Gamma)u(X_{t-})-h(X_{t-})\leq0,\\
u(A_{t}+N_{t}\Delta \xi(t))-u(A_{t})\geq \langle\deri^{1} u(A_{t}),N_{t}\rangle\Delta \xi(t).
\end{cases}
\end{equation*}
Then, 
\begin{equation*}
u(x)\leq\E_{x}\bigl(\expo^{-q(t\wedge\tau_{B_{n}(0)})}u(X_{t\wedge\tau_{B_{n}(0)}})\bigr)+\E_{x}(\int_{0}^{t\wedge\tau_{B_{n}(0)}}\expo^{-qs}(h(X_{s})\der s+\der\xi_{s})). 
\end{equation*}
Letting $n\rightarrow\infty$, it follows that  $\tau_{B_{n}(0)}\rightarrow\infty$ a.s. and hence
\begin{align}\label{demo9}
	u(x)\leq\E_{x}\bigl(\expo^{-qt}u(X_{t})\bigr)
	+\E_{x}\Biggl(\int_{0}^{t}\expo^{-qs}(h(X_{s})\der s+\der\xi_{s}) \Biggr).
\end{align}
Since $\lim_{t\rightarrow\infty}\E_{x}( \int_{0}^{t}\expo^{-qs}(h(X_{s})\der s+\der\xi_{s}))=\E_{x}( \int_{0}^{\infty}\expo^{-qs}(h(X_{s})\der s+\der\xi_{s}))$, we only need  to prove that 
\begin{equation}\label{demo9.2}
\varliminf_{t\rightarrow\infty}\E(\expo^{-qt}u(X_{t}))=0. 
\end{equation}
Assume that $\E_{x}\bigl( \int_{0}^{\infty}\expo^{-qt}h(X_{t})\der t\bigr)<\infty$. Otherwise (\ref{demo9}) is always true. This implies that $\varliminf_{t\rightarrow\infty}\E_{x}(\expo^{-qt}h(X_{t}))=0$. By (\ref{convh1}) and Taylor's Formula, we can observe that $\frac{c_{0}}{2}|y|^{2}\leq\int_{0}^{1}(1-\lambda)\langle\deri^{2}h(\lambda y)y,y\rangle\der\lambda=h(y)$. Then, using  that $u$ is a convex function and $||\deri^{1} u(y)||^{2}<1$, for all $y\in\R^{d}$, we see
\begin{align*}%\label{demo11}
	u(y)\leq u(0)+|\deri^{1} u(y)|\,|y|\leq u(0)+1+|y|^{2}\leq u(0)+1+\frac{2}{c_{0}}h(y),
\end{align*} 
for all  $y\in\R^{d}$. This implies that $ \varliminf_{t\rightarrow\infty}\E_{x}(\expo^{-qt}u(X_{t}))=0$. It follows that  $u(x)\leq V(x)$, for each $x\in\R^{d}$.
Finally, we shall show (ii). Let $x\in\R^{d}$ be an initial state and $(N^{*},\xi^{*})$ a control process such that $V_{(N^{*},\xi^{*})}(x)<\infty$, and the state process $X^{*}$ satisfies (\ref{demo10.1}). Applying similar arguments as in the previous proof of  $u\leq V$,  (\ref{dem9}) holds for $X^{*}$. From (\ref{demo10.1}), it is easily verified for $\tau^{*}_{B_{n}(0)}=\inf\{t>0 :X^{*}_{t}\notin\text{B}_{n}(0)\}$, with $n\geq1$, and $t\geq0$,  that
\begin{equation}\label{demo11}
\int_{0}^{t\wedge\tau^{*}_{B_{n}(0)}}\langle\deri^{1} u(X^{*}_{s-}),N_{s}\rangle\der\xi_{s}^{*\comp}=-\int_{0}^{t\wedge\tau^{*}_{B_{n}(0)}}\uno_{\{N^{*}_{s}=-\deri^{1} u(X^{*}_{s-})\}}\der\xi_{s}^{*\comp},
\end{equation}
and
\begin{equation}\label{demo12}
\sum_{0<s\leq t\wedge\tau^{*}_{B_{n}(0)}}\Delta\xi^{*}_{s}
=\sum_{0<s\leq t\wedge\tau^{*}_{B_{n}(0)}}(u(A_{s}+N_{s}\Delta \xi^{*}_{s})
-u(A_{s}))\uno_{\{||\Delta\xi^{*}_{s}||\neq0,\,||\Delta Y_{s}||=0\}}.
\end{equation}
Using (\ref{demo10.1}), (\ref{demo11}) and (\ref{demo12}) in (\ref{demo9}), it follows that
\begin{align}\label{demo13}
u(x)&=\E_{x}\bigl(\expo^{-q(t\wedge\tau^{*}_{B_{n}(0)})}u(X^{*}_{t\wedge\tau^{*}_{B_{n}(0)}})\bigr)+\E_{x}\Biggl(\int_{0}^{t\wedge\tau^{*}_{B_{n}(0)}}\expo^{-qs}(h(X^{*}_{s})\der s+\der\xi^{*}_{s}) \Biggr).
\end{align}
Letting $n\rightarrow\infty$ in (\ref{demo13}) and by  (\ref{demo9.2}), we get $u(x)=V_{(N^{*},\xi^{*})}(x)=V(x)$. This means that $(N^{*},\xi^{*})$ is the optimal control.
\end{proof}

\begin{proof}[Proof of Lemma \ref{convexu1.0}]
Let $X=\{X_{t}:t\geq0\}$ be the state process as in (\ref{conv3}), with $\varrho$  a control process and  $x\in B_{R}(0)$ fix an initial state. Integration by parts and It\^o's formula imply (see \cite[Cor. 2 and Thm. 33, pp. 68 and 81, respectively]{pro}) that
\begin{align}\label{expand1.0.0}
&u^{\varepsilon}(x)-\expo^{-q(t\wedge\tau_{B_{R}(0)})} u^{\varepsilon}(X_{t\wedge\tau_{B_{R}(0)}})\notag\\
  &=\int_{0}^{t\wedge\tau_{B_{R}(0)}}\expo^{-qs} (qu^{\varepsilon}(X_{s}))-\Gamma u^{\varepsilon}(X_{s})+\langle\deri^{1} u^{\varepsilon}(X_{s}),\dot{\varrho}_{s})\rangle\,\der s-M_{t\wedge\tau_{B_{R}(0)}},
\end{align}
for all $t\geq0$, with
\begin{align*}	
 M_{t}&\eqdef\int_{0}^{t}\expo^{-qs}\langle\deri^{1} u^{\varepsilon}(X_{s}),\der W_{s}\rangle\\
 &\quad+\int_{0}^{t}\int_{\R^{*}}\expo^{-qs}(E(u^{\varepsilon})(X_{s-}+z)-u^{\varepsilon}(X_{s-}))(\vartheta(\der s\times\der z)-\nu(\der z)\der s).\notag
\end{align*}
The process $M=\{M_{t}:t\geq0\}$ is a local martingale with $M_{0}=0$. Then, the process $M^{\tau_{B_{R}(0)}}\eqdef\{M_{t\wedge\tau_{B_{R}(0)}}:t\geq0\}$, is a  $\Pro_{x}$-martingale with $M_{0}=0$. Then, taking the expected value in (\ref{expand1.0.0}), it follows that
\begin{multline}\label{expand1.0.1}
u^{\varepsilon}(x)-\E_{x}(\expo^{-q(t\wedge\tau_{B_{R}(0)})} u^{\varepsilon}(X_{t\wedge\tau_{B_{R}(0)}}))\\
  =\E_{x}\biggl(\int_{0}^{t\wedge\tau_{B_{R}(0)}}\expo^{-qs} (qu^{\varepsilon}(X_{s})-\Gamma u^{\varepsilon}(X_{s})+\langle\deri^{1} u^{\varepsilon}(X_{s}),\dot{\varrho}_{s}\rangle)\der s\biggr).
\end{multline}
From (\ref{p13.0.1}), we get that
\begin{equation}\label{expand1.0.2}
\E_{x}(\expo^{-q(t\wedge\tau_{B_{R}(0)})} u^{\varepsilon}(X_{t\wedge\tau_{B_{R}(0)}}))
\geq u^{\varepsilon}(x)-\E_{x}\biggl(\int_{0}^{t\wedge\tau_{B_{R}(0)}}\expo^{-qs} (h(X_{s})+l_{\varepsilon}(\dot{\varrho}_{s}))\der s\biggr).
\end{equation}
Note that $\tau_{B_{R}(0)}<\infty$ or $\tau_{B_{R}(0)}=\infty$. On the event $\{\tau_{B_{R}(0)}<\infty\}$, we let $t\rightarrow\infty$ in (\ref{expand1.0.2}). Then,
\begin{equation}\label{expand1.0.3}
u^{\varepsilon}(x)\leq \E_{x}\biggl(\biggl(\int_{0}^{\tau_{B_{R}(0)}}\expo^{-qs} (h(X_{s})+l_{\varepsilon}(\dot{\varrho}_{s}))\der s\biggr)\uno_{\{\tau_{B_{R}(0)}<\infty\}}\biggr). 
\end{equation}
Now, on $\{\tau_{B_{R}(0)}=\infty\}$, we observe that $\expo^{-q(t\wedge\tau_{B_{R}(0)})}=0$ and $X_{t}\in B_{R}(0)$, for all $t>0$. Since $u^{\varepsilon}$ is a bounded continuous function, we have that 
\begin{equation*}
\E_{x}(\expo^{-q(t\wedge\tau_{B_{R}(0)})} u^{\varepsilon}(X_{t\wedge\tau_{B_{R}(0)}})\uno_{\{\tau_{B_{R}(0)}=\infty\}})=0. 
\end{equation*}
Then, by (\ref{expand1.0.2}), it yields that  
\begin{equation}\label{expand1.0.4}
u^{\varepsilon}(x)\leq \E_{x}\biggl(\biggl(\int_{0}^{\infty}\expo^{-qs} (h(X_{s})+l_{\varepsilon}(\dot{\varrho}_{s}))\der s\biggr)\uno_{\{\tau_{B_{R}(0)}=\infty\}}\biggr). 
\end{equation}
From (\ref{expand1.0.3}) and (\ref{expand1.0.4}), we get $u^{\varepsilon}\leq V^{\varepsilon}$. Since $\psi'_{\varepsilon}(||\deri^{1} u^{\varepsilon}(x)||^{2})\deri^{1} u^{\varepsilon}(x)$ is a Lipschitz continuous function  \cite[Thm. 6, p. 255]{pro}, the process $\widetilde{X}=\{\widetilde{X}:0\leq t\leq\tau_{B_{R}(0)}\}$ is solution to 
\begin{equation}\label{convex2}
\widetilde{X}_{t}=x+Y_{t}-\int_{0}^{t\wedge\tau_{B_{R}(0)}}2\psi'_{\varepsilon}(|\deri^{1} u^{\varepsilon}(\widetilde{X}_{s})|^{2})\deri^{1} u^{\varepsilon}(\widetilde{X}_{s})\der s, 
\end{equation}
for all $0\leq t\leq\tau_{B_{R}(0)}$. Then, its corresponding control process is given by 
\begin{equation}\label{convex4}
\dot{\varrho}^{R}_{t}=2\psi'_{\varepsilon}(|\deri^{1} u^{\varepsilon}(\widetilde{X}_{s})|^{2})\deri^{1} u^{\varepsilon}(\widetilde{X}_{s}),\ \text{for all}\ 0\leq t\leq\tau_{B_{R}(0)}.
\end{equation}
The process $\widetilde{X}$ satisfies (\ref{expand1.0.1}) and by (\ref{conv2.0.0}), from a similar argument  it follows that
\begin{equation*}
\E_{x}(\expo^{-q(t\wedge\tau_{B_{R}(0)})}u^{\varepsilon}(\widetilde{X}_{t\wedge\tau_{B_{R}(0)}}))= u^{\varepsilon}(x)-\E_{x}\biggl(\int_{0}^{t\wedge\tau_{B_{R}(0)}}\expo^{-qs}(h(\widetilde{X}_{s})+l_{\varepsilon}(\dot{\varrho}^{R}_{s}))\der s\biggr),
\end{equation*}
Proceeding in a similar way that (\ref{expand1.0.3}) and (\ref{expand1.0.4}), we have that $u^{\varepsilon}(x)= V^{\varepsilon,R}(x)$. We finish the proof.
\end{proof}

\begin{proof}[Proof of Lemma \ref{lemfrontera1.0}]
Let $\varphi$ be as in (\ref{aux.1.0}). Note that $\varphi\in\hol^{2,\alpha}(\overline{B_{R}(0)})$, since $u^{\varepsilon}\in\hol^{3,\alpha}(\overline{B_{R}(0)})$. The first and second derivatives of $\varphi$ are given by 
\begin{equation}\label{partu1.1}
\begin{cases}
\partial_{i}\varphi=2 \sum_{k}\partial_{k}u^{\varepsilon}\partial_{ki}^{2}u^{\varepsilon}-M\partial_{i}u^{\varepsilon},\\
\partial^{2}_{ij}\varphi=2 \sum_{k}(\partial^{2}_{kj}u^{\varepsilon}\partial^{2}_{ki}u^{\varepsilon}+\partial_{k}u^{\varepsilon}\partial^{3}_{kij}u^{\varepsilon})-M\partial^{2}_{ij}u^{\varepsilon}.
\end{cases}
\end{equation}
On the other hand, using (\ref{p.13.0}) and (\ref{intop4.1.0}), we get in $ B_{R}(0)$,
\begin{equation}\label{partu1.0}
\begin{cases}
  -\frac{M}{2}  \sum_{ij}\sigma_{ij}\partial_{ij}^{2} u^{\varepsilon}&=M(h-q'u^{\varepsilon}-\psi_{\varepsilon}(g)+\langle\deri^{1}u^{\varepsilon},\widetilde{\gamma}\rangle+\inted E(u^{\varepsilon})),\\
  \frac{1}{2} \sum_{kij}\sigma_{ij}\partial_{k}u^{\varepsilon}\partial^{3}_{kij}u^{\varepsilon}&=q'|\deri^{1}u^{\varepsilon}|^{2} -\langle\deri^{1}u^{\varepsilon},\deri^{1}h\rangle+\psi'_{\varepsilon}(g)\langle\deri^{1} u^{\varepsilon},\deri^{1} g\rangle\\
  &\quad-\langle\deri^{2} u^{\varepsilon}\deri^{1}u^{\varepsilon},\widetilde{\gamma}\rangle- \sum_{i}\partial_{i}u^{\varepsilon} \inte E(\partial_{i}u^{\varepsilon}),
\end{cases}
\end{equation}
where  the first and second derivatives of $g$ are given in (\ref{intop4.1.3.0}). Then, from \eqref{partu1.1}--\eqref{partu1.0}, we see that
\begin{align}\label{partu1}
\frac{1}{2}\sum_{ij}\sigma_{ij}\partial_{ij}^{2}\varphi&=\sum_{kij}\sigma_{ij}\partial^{2}_{kj}
u^{\varepsilon}\partial^{2}_{ki}u^{\varepsilon}-2\langle\deri^{1}u^{\varepsilon},\deri^{1}h\rangle +Mh+q'(2|\deri^{1}u^{\varepsilon}|^{2}-Mu^{\varepsilon})\notag\\
&\quad+2\psi'_{\varepsilon}(g)\langle\deri^{1} u^{\varepsilon},\deri^{1} g\rangle-M\psi_{\varepsilon}(g)+M \inte E(u^{\varepsilon})\notag\\
&\quad-2 \sum_{i}\partial_{i}u^{\varepsilon}\inte E(\partial_{i}u^{\varepsilon})-2\langle\deri^{2} u^{\varepsilon}\deri^{1}u^{\varepsilon},\widetilde{\gamma}\rangle+M\langle\deri^{1}u^{\varepsilon},\widetilde{\gamma}\rangle.
\end{align}
Lemma \ref{inde} implies 
\begin{equation}\label{partu12}
q'(2|\deri^{1}u^{\varepsilon}|^{2}-Mu^{\varepsilon})\geq-Mq'K_{5}.
\end{equation}
The constant $K_{5}$ is as in Lemma \ref{inde}. By (H\ref{h1}) and (H\ref{h3}), it follows
\begin{equation}\label{partu2}
\sum_{kij}\sigma_{ij}\partial^{2}_{kj}u^{\varepsilon}\partial^{2}_{ki}u^{\varepsilon}-2\langle\deri^{1}u^{\varepsilon},\deri^{1}h\rangle+Mh\geq-2C_{0}|\deri^{1}u^{\varepsilon}|,
\end{equation}
The constants $\theta$ and $C_{0}$ are given in (H\ref{h1}) and (H\ref{h3}), respectively. Since  $\psi_{\varepsilon}(g)\leq\psi'_{\varepsilon}(g)g$ and $\partial_{i}\varphi=2\sum_{k}\partial_{k}u^{\varepsilon}\partial_{ki}^{2}u^{\varepsilon}-M\partial_{i}u^{\varepsilon}$ for all $i\in\{1,\dots,d\}$, we have that 
\begin{equation}\label{partu13}
\begin{cases}
\langle\deri^{1}\varphi,\widetilde{\gamma}\rangle=2\langle\deri^{2} u^{\varepsilon}\deri^{1}u^{\varepsilon},\widetilde{\gamma}\rangle-M\langle\deri^{1}u^{\varepsilon},\widetilde{\gamma}\rangle,\\
2\psi'_{\varepsilon}(g)\langle\deri^{1} u^{\varepsilon},\deri^{1} g\rangle-M\psi_{\varepsilon}(g)\geq\psi'_{\varepsilon}(g)(2\langle\deri^{1}\varphi,\deri^{1}u^{\varepsilon}\rangle+M|\deri^{1}u^{\varepsilon}|^{2}).
\end{cases}
\end{equation}
Since $\int_{|x+z|\leq R} u^{\varepsilon}(x+z)\nu(\der z)\geq0$ and from Lemmas \ref{cotaintaf1}-\ref{intcotadif1.1}, it follows that
\begin{multline}\label{partu3}
M \inte E(u^{\varepsilon})-2 \sum_{i}\partial_{i}u^{\varepsilon} \inte E(\partial_{i}u^{\varepsilon})\\
\geq -2MA_{0}K_{5}\nu(\mathcal{B}')-2dC_{1}K_{5}\nu_{0}|\deri^{1}u^{\varepsilon}|-2dM(\nu_{0}+dC_{1}\nu(\mathcal{B}'))|\deri^{1}u^{\varepsilon}|.
\end{multline}
The constants $\nu_{0}$, $A_{0}$ and $C_{1}$ are  as in (H\ref{h2}) and Lemmas \ref{cotaintaf1}--\ref{intcotadif1.1}, respectively, and $\mathcal{B}'= B_{R+\frac{b}{2}}(0)\setminus\overline{B_{R}(0)}$. Defining 
\begin{align*}
K_{7}&\eqdef K_{5}(q'+2A_{0}\nu(\mathcal{B}')),\\ 
K_{8}&\eqdef2(C_{0}+dC_{1}K_{5}\nu(\mathcal{B}')),\\
K_{9}&\eqdef 2d(\nu_{0}+dC_{1}\nu(\mathcal{B}')), 
\end{align*}
which are independent of $\varepsilon$, and applying (\ref{partu12})--(\ref{partu3}) in (\ref{partu1}), it yields (\ref{partu4}). 
\end{proof}

\begin{proof}[Proof of Lemma \ref{cotaphi.1.0}]
Let $B_{r}\subset B_{R}(0)$ be an open ball. For each $\xi\in\hol_{\comp }^{\infty}(B_{r})$ satisfying $0\leq\xi\leq1$, define $\phi$ as in (\ref{phi.1.0}). The  first and second derivatives of $\phi$ in $B_{r}$ are given by
\begin{equation}\label{psi1}
\begin{cases}
  \partial_{i} \phi&=\psi_{\varepsilon}(g)\partial_{i} \xi+\xi\psi'_{\varepsilon}(g)\partial_{i}g,\\
  \partial^{2}_{ji} \phi&=\psi_{\varepsilon}(g)\partial^{2}_{ji} \xi+\psi'_{\varepsilon}(g)\partial_{i} \xi\partial_{j}g+\xi\psi''_{\varepsilon}(g)\partial_{j}g\partial_{i}g\\
  &\quad+\psi'_{\varepsilon}(g)\partial_{j}\xi\partial_{i}g+\xi\psi'_{\varepsilon}(g)\partial^{2}_{ji} g,
 \end{cases}
\end{equation}
where  the first and second derivatives of $g$ are given in (\ref{intop4.1.3.0}). Then, by (\ref{partu1.0}) and \eqref{psi1}, we get
\begin{align}\label{psi1.2}
\frac{1}{2}\sum_{ji}\sigma_{ji}\partial^{2}_{ji}\phi&=\frac{\psi_{\varepsilon}(g)}{2}\sum_{ji}\sigma_{ji}\partial^{2}_{ji} \xi+\frac{\xi\psi''_{\varepsilon}(g)}{2}\sum_{ji}\sigma_{ji}\partial_{j}g\partial_{i}g+\psi'_{\varepsilon}(g)\biggl(\sum_{ji}\sigma_{ji}\partial_{i} \xi\partial_{j}g\notag\\
&\quad+\xi\sum_{jik}\sigma_{ji}\partial^{2}_{kj}u^{\varepsilon}\partial^{2}_{ki}u^{\varepsilon}+2\xi\biggr(q'|\deri^{1}u^{\varepsilon}|^{2} -\langle\deri^{1}u^{\varepsilon},\deri^{1}h\rangle\notag\\
&\quad+\psi'_{\varepsilon}(g)\langle\deri^{1} u^{\varepsilon},\deri^{1} g\rangle-\langle\deri^{2} u^{\varepsilon}\deri^{1}u^{\varepsilon},\widetilde{\gamma}\rangle-\displaystyle\sum_{i}\partial_{i}u^{\varepsilon}\inte E(\partial_{i}u^{\varepsilon})\biggl)\biggr).
\end{align}
On the other hand, applying Hypothesis (H\ref{h1}) and Lemmas \ref{cotaintaf1}, \ref{inde} and \ref{lemafrontera1},  in (\ref{p.13.0}), it yields
\begin{equation}\label{cotgrad1}
\psi_{\varepsilon}(g)\leq \frac{d^{2}|\sigma|\,|\deri^{2}u^{\varepsilon}|}{2}+C_{0}+K_{10}|\widetilde{\gamma}|+K_{5}(\nu_{0}+2A_{0}\nu(\mathcal{B}')).
\end{equation}
From Hypothesis (H\ref{h3}), \eqref{cotgrad1}, $\xi\geq0$ and $\psi''_{\varepsilon}\geq0$,  we have 
\begin{multline}\label{psi2}
\frac{\psi_{\varepsilon}(g)}{2}\tr(\sigma \deri^{2}\xi)+\frac{\xi\psi''_{\varepsilon}(g)}{2}\langle \sigma\deri^{1}g,\deri^{1}g\rangle\\
\geq-\frac{d^{2}|\deri^{2}\xi|\,|\sigma|}{2}\biggl(\frac{d^{2}|\sigma|\,|\deri^{2}u^{\varepsilon}|}{2}+C_{0}+K_{10}|\widetilde{\gamma}|+K_{5}(\nu_{0}+2A_{0}\nu(\mathcal{B}'))\biggr).
\end{multline}
Using Hypotheses (H\ref{h1}), (H\ref{h3}) and Lemma \ref{lemafrontera1}, it implies
\begin{equation}\label{psi3}
\begin{cases}
q'|\deri^{1}u^{\varepsilon}|^{2} -\langle\deri^{1}u^{\varepsilon},\deri^{1}h\rangle\geq -C_{0}K_{10},\\
\sum_{ji}\sigma_{ji}\partial_{i} \xi\partial_{j}g+\xi\sum_{jik}\sigma_{ji}\partial^{2}_{kj}u^{\varepsilon}\partial^{2}_{ki}u^{\varepsilon}\\
\hspace{3.5cm}\geq-2d^{3}K_{10}|\sigma|\,|\deri^{1}\xi|\,|\deri^{2}u^{\varepsilon}|+\theta\xi|\deri^{2}u^{\varepsilon}|^{2}.
\end{cases}
\end{equation}
From \eqref{psi1}, (\ref{cotgrad1}) and Lemma \ref{lemafrontera1}, it yields
\begin{multline}\label{psi5}
\xi\psi'_{\varepsilon}(g)\langle\deri^{1}g,\deri^{1}u^{\varepsilon}\rangle\geq\langle\deri^{1}\phi,\deri^{1}u^{\varepsilon}\rangle\\
-K_{10}|\deri^{1}\xi|\biggl(\frac{d^{2}|\sigma|\,|\deri^{2}u^{\varepsilon}|}{2}+C_{0}+K_{10}|\widetilde{\gamma}|+K_{5}(\nu_{0}+2A_{0}\nu(\mathcal{B}'))\biggr).
\end{multline}
Finally, Lemmas \ref{intcotadif1.1}, \ref{inde} and \ref{lemafrontera1} imply
\begin{equation}\label{psi6}
-\langle\deri^{2}u^{\varepsilon}\deri^{1}u^{\varepsilon},\widetilde{\gamma}\rangle-\sum_{i}\partial_{i}u^{\varepsilon}\inte E(\partial_{i}u^{\varepsilon})\\
\geq-K_{10}|\widetilde{\gamma}|\,|\deri^{2}u^{\varepsilon}|-dK_{10}(K_{10}\nu_{0}+C_{1}\nu(\mathcal{B}')(K_{5}+dK_{10})).
\end{equation}
Defining the constants
\begin{align*}
K_{11}&\eqdef\frac{d^{2}|\sigma|\,|\deri^{2}\xi|}{2},\ K_{12}\eqdef\frac{d^{2}|\sigma|}{2},\\
K_{13}&\eqdef C_{0}+|\widetilde{\gamma}|K_{10}+K_{5}(\nu_{0}+2A_{0}\nu(\mathcal{B}')),\\
K_{14}&\eqdef 2K_{10}(|\deri^{1}\xi|(d^{3}|\sigma|+K_{12})+|\widetilde{\gamma}|),\\
K_{15}&\eqdef 2K_{10}(K_{13}|\deri^{1}\xi|+C_{0}+d(K_{10}\nu_{0}+C_{1}\nu(\mathcal{B}')(K_{5}+dK_{10}))),
\end{align*}
which are independent of $\varepsilon$, and applying (\ref{psi2})--(\ref{psi6}) in (\ref{psi1.2}), we conclude \eqref{psic1}.
\end{proof}

\begin{proof}[Proof of Lemma \ref{converu1}]
Let $d< p<\infty$, $B_{r}\subset B_{R}(0)$ an open ball and $\beta\in(0,1)$ fixed. Since the sequence  $\{u^{\varepsilon}\}_{\varepsilon\in(0,1)}$ satisfies (\ref{sob1}) and \eqref{hol.1.1.0.1}, by Arzel\`a-Ascoli Theorem; see \cite[Thm. 7.25, p. 158]{rudin}, it follows there exist a decreasing subsequence $\{\varepsilon_{\kappa_{1}(\iota)}\}_{\kappa_{1}(\iota)\geq1}$, with $\varepsilon_{\kappa_{1}(\iota)}\underset{\iota\rightarrow\infty}{\longrightarrow}0$, and  $u\in\hol^{0,1}(\overline{B_{R}(0)})\cap\hol^{1}(B_{R}(0))$ such that 
\begin{equation}\label{limit.1.0}
\begin{cases}
u^{\varepsilon_{\kappa_{1}(\iota)}}\underset{\varepsilon_{\kappa_{1}(\iota)}\rightarrow0}{\longrightarrow}u & \text{in}\ \hol^{0}(\overline{B_{R}(0)}),\\
\partial_{i}u^{\varepsilon_{\kappa_{1}(\iota)}}
\underset{\varepsilon_{\kappa_{1}(\iota)}\rightarrow0}{\longrightarrow}\partial_{i}u, &\text{in}\ \hol^{0}_{\loc}(B_{R}(0)). 
\end{cases}
\end{equation}
Now, observing that the subsequence  $\{u^{\varepsilon_{\kappa_{1}(\iota)}}\}_{\kappa_{1}(\iota)\geq1}$ satisfies \eqref{sob2}, from \eqref{limit.1.0} and the reflexivity of $\Lp^{p}(B_{r})$; see \cite[Thm. 2.46, p. 49]{adams}, we have that there exists a subsequence $\{u^{\varepsilon_{\kappa_{2}(\iota)}}\}_{\kappa_{2}(\iota)\geq1}$ of $\{u^{\varepsilon_{\kappa_{1}(\iota)}}\}_{\kappa_{1}(\iota)\geq1}$ such that 
\begin{equation*}
\partial^{2}_{ij}u^{\varepsilon_{\kappa_{2}(\iota)}}\underset{\varepsilon_{\kappa_{2}(\iota)}\rightarrow0}{\longrightarrow}\partial^{2}_{ij}u,\ \text{weakly in}\ \Lp^{p}(B_{\beta r}), 
\end{equation*}
where $\partial^{2}_{ij}u$ represents the second weakly derivative of $u$, with $i,j\in\{1,\dots,d\}$. We shall show below (\ref{conver1.0.3}). Note that for each $x\in B_{R}(0)$, by Proposition \ref{cotEu1}, we have 
\begin{equation*}
|\inted E(u^{\varepsilon_{\kappa_{2}(\iota)}})(x)-\inted E (u)(x)|\leq2A_{0}\nu_{0}||u^{\varepsilon_{\kappa_{2}(\iota)}}-u||_{\hol^{0}(\overline{B_{R}(0)})}\underset{\varepsilon_{\kappa_{2}(\iota)}\rightarrow0}{\longrightarrow}0, 
\end{equation*}
and hence $\inted E(u^{\varepsilon_{\kappa_{2}(\iota)}})\underset{\varepsilon_{\kappa_{2}(\iota)}\rightarrow0}{\longrightarrow}\inted E (u)$, uniformly in $B_{R}(0)$. We conclude that there exist a decreasing subsequence $\{\varepsilon_{\kappa(\iota)}\}_{\iota\geq1}$, with $\varepsilon_{\kappa(\iota)}\underset{\iota\rightarrow\infty}{\longrightarrow}0$, and  $u\in\hol^{0,1}(\overline{B_{R}(0)})\cap\sob^{2,p}(B_{\beta r})$ satisfying  (\ref{conver.1.0}) and (\ref{conver1.0.3}).
  \end{proof}
\appendix
\section*{Acknowledgments}
The results in this paper are part of the Ph.D. thesis of the author H. A. 
Moreno-Franco \cite{moreno}, under the supervision  of Dr. Daniel
Hern\'andez-Hern\'andez and Dr. V\'ictor Rivero. The author would like to  
thank: CONACyT and CIMAT for the Ph.D. fellowship and facilities provided; 
National Research University Higher School of Economics for the financial 
support in finishing this project; his doctoral advisors of thesis 
Dr. Daniel Hern\'andez-Hern\'andez and Dr. V\'ictor Rivero, 
for their guidance on this work; and finally, his readers of thesis Dr.
Jose Luis Menaldi, Dr. Renato Iturriaga, Dr. Hector Sanchez and Dr. Juan 
Carlos Pardo, for their advice and suggestions.
%\begin{acknowledgements}
%\end{acknowledgements}



\end{document}